\documentclass[preprint,10pt]{elsarticle}

\graphicspath{./figures}
\pdfoutput=1

\usepackage{amsmath,amssymb,amsthm}
\allowdisplaybreaks[4]  
\usepackage{hyperref}
\usepackage{cleveref}
\usepackage{graphicx}
\usepackage{graphics}
\usepackage {graphicx,fancyhdr}
\usepackage{graphics}
\usepackage{color}
\usepackage{flafter}
\usepackage{multirow}
\usepackage{mathrsfs} 
\usepackage{subfigure}
\usepackage{tabularx}
\usepackage{lineno}
\usepackage{xfrac}
\usepackage{rotating} 
\usepackage{stmaryrd} 
\usepackage{bm} 
\usepackage{cases} 
\usepackage{enumerate}
\usepackage{ulem}  
\usepackage{url}

\modulolinenumbers[5]  

\newcommand*{\snorm}[1]{\lvert #1 \rvert} 
\newcommand*{\dnorm}[1]{\lVert #1\rVert} 
\newcommand*{\jump}[1]{\llbracket #1\rrbracket} 
\newcommand*{\hhoVecvars}[1]{\underline{\bm{#1}}}

\newcommand{\transpose}{\intercal} 
\newcommand{\realR}{\ensuremath{\mathbb{R}}}
\newcommand{\rd}{\ensuremath{\mathbb{R}^d}}

\newcommand{\sqd}[1]{\ensuremath{[{#1}]^d}}  
\newcommand{\mcalTh}{\mathcal{T}_h}
\newcommand{\mcalFh}{\mathcal{F}_h}
\newcommand{\mcalT}{\mathcal{T}}
\newcommand{\mcalF}{\mathcal{F}}
\newcommand{\mbbP}{\mathbb{P}}
\newcommand{\matvar}[1]{\ensuremath{\uuline{\bm{#1}}}}

\newcommand{\bcurl}{\ensuremath{{\bf{curl}\hspace{0.1em}}}}
\newcommand{\identityM}{\ensuremath{{\bf{I}}}}
\newcommand{\bdeltaT}{\ensuremath{\bm{\delta}_T^k}}
\newcommand{\bdeltaF}{\ensuremath{\bm{\delta}_{TF}^k}}
\newcommand{\comment}[1]{\iffalse #1 \fi }

\theoremstyle{plain}

\newtheorem{Theorem}{Theorem}[section]
\newtheorem{Lemma}[Theorem]{Lemma}

{\normalfont}

\begin{document}
	\begin{frontmatter}
		\title{A posteriori error analysis of hybrid high-order method for the Stokes problem} 
		
		\author[label1]{Yongchao Zhang} \ead{yoczhang@nwu.edu.cn}
		\author[label2]{Liquan Mei} \ead{lqmei@mail.xjtu.edu.cn}
		\author[label3]{Gang Wang\corref{label4}} \ead{gangwang.math@nwpu.edu.cn}
		
		\address[label1]{School of Mathematics, Northwest University, Xi'an, Shaanxi 710069, P. R. China}
		\address[label2]{School of Mathematics and Statistics, Xi'an Jiaotong University, Xi'an, Shaanxi 710049, P. R. China}
		\address[label3]{School of Mathematics and Statistics, Northwestern Polytechnical University, Xi'an, Shaanxi 710062, P. R. China}
		
		\cortext[label4]{Corresponding author.}
		
		\begin{abstract}
			We present a residual-based a posteriori error estimator for the hybrid high-order (HHO) method for the Stokes model problem. Both the proposed HHO method and error estimator are valid in two and three dimensions and support arbitrary approximation orders on fairly general meshes. The upper bound and lower bound of the error estimator are proved, in which proof, the key ingredient is a novel stabilizer employed in the discrete scheme. By using the given estimator, adaptive algorithm of HHO method is designed to solve model problem. Finally, the expected theoretical results are numerically demonstrated on a variety of meshes for model problem.
		\end{abstract}
		
	\end{frontmatter}

	\section{Introduction}
	In recent years, the adaptive mesh refinement methods for solving both linear and nonlinear problems have been the object of intense study, since it can derive more uniform error distribution on each element of the mesh partition, while providing the desired accuracy and minimum cost. In particular, these methods are very effective for problems with exhibits singularities or strong geometrical localized variations. The main feature of the adaptive procedure is justified by using a posteriori error estimate to provide computable lower and upper error bounds, which serves then as error indicators on elements for the adaptive mesh coarsening or refinement \cite{1994_JCAM_adaptive}. A posteriori error analysis of the finite element method (FEM) for second-order elliptic equation can be traced back to \cite{1978_JNME_Babuska,1978_SIAM_Babuska}, and then a large number of related studies emerged, such as \cite{1994_JCAM_adaptive,1997_CMAME_Ainsworth,2005_CS_Thomas,2012_ANM_nonconforming}. However, the classical FEM cannot work well on general polygonal elements and hanging nodes are not allowed, which means that local mesh adaption requires special strategies to either treat or prevent hanging nodes. This drawback also limits the application of FEM adaptive algorithms, especially for complex geometries or problems with certain physical constraints, some such examples can be found in \cite{1995_CMAME_Somnath,2011_book_Micromechanical,2015_book_Tissue}. 
	
	To get more flexible mesh generation and adaptation, the numerical methods on general meshes have received significant attention over the last decade. Novel approaches have been rediscovered or developed on general polytopal meshes, here, we mention, for example, discontinuous Galerkin (DG) method \cite{2010_MC_DGNS}, hybrid discontinuous Galerkin (HDG) method \cite{2016_chapter_Static,2016_ESAIM_Bridging}, weak Galerkin (WG) method \cite{2015_IJNAM_WGpoly}, virtual element method (VEM) \cite{2013_VEMBasic_Brezzi}, hybrid high-order (HHO) method \cite{2014_CMAM_Pietro}, etc., for which, the corresponding adaptive algorithms have also been well constructed, such as \cite{2014_M3AS_hpDG,2013_SIAM_posterioriHDG,2014_JSC_posterioriWG,2015_ESAIM_posterioriVEM,2016_JCP_posterioriHHO}. Our focus is here on the HHO method, which was originally introduced in \cite{2014_CMAM_Pietro,2015_CMAME_hhoelasticity} to solve the quasi-incompressible linear elasticity models. On each element, the HHO discretization space hinges on local reconstruction operators from hybrid polynomial unknowns at the interior and faces of the element. Benefiting from the local discretization space, the HHO method supports general meshes in both two and three dimension and arbitrary polynomial order $ k\geq 0 $. The HHO framework has been used to solve various PDEs, a non-exhaustive list includes \cite{2017_DiPietro_LerayLions,2018_SIAMJSC_Chave,2018_JSC_hhoNS,2020_JCP_ZhangSD}.
	
	The steady Stokes model problem is considered in this paper, which describes the viscous incompressible flow with high viscosity. A posteriori error analysis of the model problem, for different numerical methods, has been the subject of various investigations and we refer to the pioneer works, e.g., \cite{1991_SIAM_Bank,2005_MC_Ainsworth,2012_NM_unified,2019_SIAM_BaoWG,2020_JSC_VEMPostStokes} for more detailed reviews on the subject. Remarkably, in \cite{2012_NM_unified}, based on $ \sqd{H_0^1(\Omega)} $-conforming velocity reconstruction and locally conservative flux reconstruction, the authors gave a unified framework for a posteriori error estimation for the Stokes problem and applied this framework to the conforming, nonconforming, mixed FEMs, the DG method and a general class of finite volume methods. As for the HHO method, several results of a posterior error analysis have been available. The authors in \citep[Chapter 4]{2020_hho_book} derived a posteriori error estimates for the Poisson equation, a reliable upper bound and efficient lower bound of the error in terms of residual-based estimators were proved. In \cite{2016_JCP_posterioriHHO}, a posteriori error estimate for the mixed HHO method was proposed and an adaptive resolution algorithm for problems in electrostatics was presented by using the corresponding estimators in the context. 
	
	As far as we know, there are few studies of a posteriori error analysis of the HHO method for the Stokes model problem. The main result of this work is that, for the model problem, we devise the new a posteriori error estimator and give the posteriori-based adaptive HHO algorithm for model problem. Here, we mention that the key point to use the abstract error estimate in this paper is a new reformulation of the stabilization term in the HHO discretization scheme. We prove that the stabilization term is polynomial consistency and leads to the corecivity of the velocity discretization, which ensures the well-posedness of the discrete problem and local lower bound of the error estimator. As a consequence, we establish a simple a posteriori error analysis in our context, which bases on the Helmholtz decomposition of the error and the discrete inf-sup condition. Similar techniques have been employed in \cite{1995_MC_nonc_Stokes} for nonconforming FEM method and in \cite{2019_SIAM_BaoWG,2017_JAM_Xiexiaoping} for weak Galerkin method. The a posteriori error estimator contains of the stabilizer, the divergence and the jump of discrete velocity reconstruction, in addition, both upper and lower bounds involve the data oscillation. Furthermore, all terms can be computed locally on general mesh elements in two or three dimensions.
	
	The rest of this paper is organized as follows. In Section \ref{sec_preliminaries} we introduce the Stokes model problem, give the notations for continuous and discrete settings, and recall some basic results on broken polynomial spaces. In Section \ref{sec_discretization} we define the discrete spaces for velocity and pressure, and establish the local reconstructions and discrete problem. Moreover, the errors for velocity and pressure are defined, and the convergence analysis is carried out in Lemma \ref{sec_model_lem2}. Section \ref{sec_posterioriError} collects the a posteriori error estimator, some technical lemmas relevant to the analysis of the upper and lower bounds, which are also the main theorems of this section. In Section \ref{num_exams} we present some numerical experiments to validate the effectiveness of the estimator and theoretical results. Conclusions and perspectives are discussed in Section \ref{Conclusion}.

	\comment{
		
		The key point to derivate the upper bound of the estimators in our context is a new reformulation of the stabilization term in the HHO discretization scheme.
		
		We prove that the stabilization ensures the coercivity 
		
		Here, the similar stabilization term has been mentioned in 
		
		a posteriori-based adaptive HHO algorithm devised for problems in electrostatics. 
		
		2016_JCP_posterioriHHO : 
		The main results of this work are new a posteriori error estimate for the MHO method and the use of the corresponding estimators in the context of an adaptive resolution algorithm for problems in electrostatics. For the derivation of our error estimators, we use a residual-based approach that relies on an abstract estimate inspired by [22]; cf. also [23]. We also cite [24] for the idea of relying on the equivalent primal formulation to derive a posteriori bounds for mixed methods. 
		
		A key point to use the abstract error estimate in our context is a new reformulation of the stabilization term in the MHO method, which establishes a new connection with the hybridized version studied in [21]. 
		
		The upper bound thus derived has no undetermined constants and, possibly up to minor modifications, also extends to sibling schemes such as the Hybrid High-Order method of [25] (cf. also [26]) and the DGA method of [20]. Adaptive resolution methods for related high-order polyhedral discretization methods, each with its own distinguishing features, are developed in [27–29].
	}
	
	\comment{
		then analysis of nonconforming FEM in [] followed.
		
		A posteriori error analysis for the numerical solution of differential equations started in late 1970s. 
		
		二阶椭圆方程的有限元方法的后验误差分析最早可以追溯到 [1,2].
		
		A posteriori error analysis for the numerical solution of differential equations started in late 1970s with the seminal work by Babuška and Rheinboldt [1,2]. 
		
		The research in \cite{1978_JNME_Babuska,1978_SIAM_Babuska} provided the earliest mathematical theory for a posteriori error estimation by using conforming finite element method (FEM), and the nonconforming FEM 
		
		A posteriori error analysis of finite element methods for the model problem 
	}
	
	\comment{ 
		(2015 A posteriori error analysis for Navier–Stokes equations coupled with Darcy problem)
		(JingFei  JCAM posteriori error)

		relies on the numerical solution and the data of the model problem. The 
		
		provides computable upper and lower error bounds. 
		
		to find a reliable and efficient a posteriori error estimator, which relies on the numerical solution and the data of the model problem.

		The basis of a successful adaptive procedure is the availability of a reliable and efficient a posteriori error estimator of a numerical solution, which is a computable quantity depending only on the problem data and the numerical solution. 
		
		The key features to these methods are a posteriori error analysis leading to a mesh refinement strategy. In fact, a posteriori error estimators are quantities that can be easily calculated from the computed numerical solution and the data of the problem. 
		
		Because we cannot directly calculate the error of the numerical solution, to find reliable and efficient a posteriori error estimators is the key to the adaptive procedure. The computability of a posteriori error estimator relies on the numerical solution and the data of the problem. Based on reliable and efficient error indicators on elements, the adaptive mesh refinement or coarsening can be implemented [23].
		
		The key of this technique is to design an efficient and reliable a posteriori error estimator using the numerical solution and the known data like the load term. 
		
		In this paper, we derive a residual-based energy norm a posteriori error estimator for the steady Stokes model problem, which describes the viscous incompressible flow. 
		
		In particular, it is very effective for problems with strong geometrical localized variations or exhibits singularities, 
		
		Especially, it is very efficient for problems with strong geometrical localized variations or exhibits singularities since it can derive more uniform error distribution on each element of the domain and provide both desired accuracy and minimum cost at the same time. 
		
		The basis of a successful adaptive procedure is the availability of a reliable and efficient a posteriori error estimator of a numerical solution, which is a computable quantity depending only on the problem data and the numerical solution. 
	}
	
	\section{Preliminaries} \label{sec_preliminaries}
	In this section, we introduce the model problem and the settings of continuous Sobolev spaces. Then, we recall the notions of discrete mesh and give some basic results on the local broken polynomial spaces.

	\subsection{Model problem}\label{sec_model_subsec1}
	In this paper, we consider the following Stokes problem on a bounded and simply connected domain $ \Omega\subset \rd $ ($ d=2,3 $): given the body force $ \bm{f}=\big(f_i(x_1,...,x_d)\big)_{1\leq i\leq d}: \Omega\rightarrow \rd $, find the fluid velocity $ \bm{u}=\big(u_i(x_1,...,x_d)\big)_{1\leq i\leq d}: \Omega\rightarrow \rd $ and pressure $ p(x_1,...,x_d): \Omega\rightarrow \realR $, such that
	\begin{subequations}\label{sec_model_eq1}
		\begin{align}
			\label{sec_model_eq1a}
			-\nu \Delta \bm{u} + \nabla p &= \bm{f} \quad \text{in}\ \Omega, \\ 
			\label{sec_model_eq1b}
			\nabla\cdot\bm{u} &= 0 \quad \,\text{in}\ \Omega, \\ 
			\label{sec_model_eq1c}
			\bm{u} &= \bm{0} \quad \text{on}\ \partial \Omega, 
		\end{align}
	\end{subequations}
	where $ \nu>0 $ denotes the kinematic viscosity and, the Laplace operators $ \Delta=:\sum_{i=1}^d\partial_{x_i x_i} $, gradient operator $ \nabla =:(\partial_{x_i})_{1\leq i\leq d} $, divergence operator $ \nabla\cdot $, are such that $ (\Delta\bm{u})_{j\, (1\leq j\leq d)} = \sum_{i=1}^d\partial_{x_i x_i} u_j $, $ (\nabla \bm{u})_{ij\, (1\leq i, j\leq d)} = \partial_{x_i}u_j $, $ \nabla\cdot\bm{u} = \sum_{i=1}^d\partial_{x_i} u_i $, respectively.

	Throughout this paper, for a bounded subset $ D $ in $ \mathbb{R}^n $ ($ n=d, d-1 $), we adopt the notations $ W^{s,r}(D) $ ($ s, r \geq 0 $) to indicate the standard Sobolev space equipped with the norm $ \dnorm{\,\cdot\,}_{s,r,D} $ and seminorm $ \snorm{\,\cdot\,}_{s,r,D} $. For $ s>0, r=2 $, we use the notation $ H^{s}(D):=W^{s,2}(D) $ with norm $ \dnorm{\,\cdot\,}_{s,D} $ and seminorm $ \snorm{\,\cdot\,}_{s,D} $ accordingly. The space $ W^{0,r}(D) $ coincides with $ L^r(D) $, in particular, $ W^{0,2}(D) = L^2(D) $, for which the norm and inner product are represented by $ \dnorm{\,\cdot\,}_D $ and $ (\cdot,\cdot)_D $, respectively. When $ D=\Omega $, we will omit the subscript $ D $ in the norms and inner product. By a slight abuse of notation, the above (semi)norm notations are also applicable to vector spaces $ [W^{s,r}(D)]^n $ (and matrix spaces $ [W^{s,r}(D)]^{n\times n} $). Moreover, another matrix space is defined by 
	\begin{align}\label{sec_model_eq1_add1}
		\bm{H}({\rm div};D):=\Big\{ \bm{\tau}\in [L^2(D)]^{n\times n}: \nabla\cdot \bm{v} \in [L^2(D)]^n \Big\}
	\end{align}
	with the norm $ \dnorm{\bm{\tau}}^2_{{\rm div}, D}:= \dnorm{\bm{\tau}}_D^2 + \dnorm{\nabla\cdot\bm{\tau}}_D^2 $.
	
	Next, we assume the body force $ \bm{f}\in\sqd{L^2(\Omega)} $, and define the spaces for the velocity and pressure as follows:
	\begin{align*}
		\bm{V}:= \sqd{H_0^1(\Omega)}=\Big\{\bm{v}\in \sqd{H^1(\Omega)}: \bm{v}|_{\partial\Omega}=\bm{0}\Big\}, \qquad Q:=L_0^2(\Omega) = \Big\{q\in L^2(\Omega): \int_\Omega q = 0\Big\} .
	\end{align*}
	A standard weak form in the primary velocity-pressure formulation for the model problem \eqref{sec_model_eq1a}$ - $\eqref{sec_model_eq1c} reads: find $ \bm{u}\in \bm{V} $ and $ p\in Q $ such that
	\begin{subequations}\label{sec_model_eq2}
		\begin{align}
			\label{sec_model_eq2a}
			\nu a(\bm{u},\bm{v}) + b(\bm{v},p) &= (\bm{f},\bm{v}) \hspace{1.74em} \forall \bm{v} \in\bm{V}, \\
			\label{sec_model_eq2b}
			-b(\bm{u},q) &= 0 \hspace{3.72em} \forall q\in Q,
		\end{align}
	\end{subequations}
	with bilinear forms $ a: \bm{V}\times\bm{V}\rightarrow \realR $ and $ b: \bm{V}\times Q\rightarrow\realR $ defined by 
	\begin{align}\label{sec_model_eq3}
		a(\bm{w},\bm{v}):=(\nabla\bm{w},\nabla\bm{v}), \quad b(\bm{v}, q):=-(\nabla\cdot\bm{v}, q).
	\end{align}
	
	It is well-known that the well-posedness of weak formulation \eqref{sec_model_eq2a}$ - $\eqref{sec_model_eq2b} (\citep[Theorem 4.6]{2016_JohnVolker_book}) hinges on the coercivity of the bilinear form $ a $ together with the inf-sup condition
	\begin{align}\label{sec_model_eq4}
		\beta \leq \inf_{q\in Q\backslash\{0\}}\sup_{\bm{v}\in\bm{V}\backslash\{\bm{0}\}}\frac{b(\bm{v},q)}{\dnorm{\nabla\bm{v}}\dnorm{q}},
	\end{align}
	where $ \beta $ is a positive constant.

	\subsection{Meshes and basic results}\label{sec_model_subsec2}
	We recall the mesh-related notations and some basic results on broken polynomial spaces from \cite{2020_hho_book,2018_JSC_hhoNS}.
	
	Let $ \mathcal{H}\subset (0, +\infty) $ denote a countable set of meshsizes having $ 0 $ as its unique accumulation point. We consider $ h $-refined mesh sequences $ (\mathcal{T}_h)_{h\in\mathcal{H}} $, and for all $ h\in\mathcal{H} $, $ \mathcal{T}_h $ is a finite collection of nonempty disjoint open polygonal (or polyhedral) elements $ T $ with diameter $ h_T $ such that $ \overline{\Omega} = \bigcup_{T\in\mathcal{T}_h} \overline{T} $ and $ h=\max_{T\in\mathcal{T}_h} h_T $. The set of mesh faces $ \mathcal{F}_h $ is a finite collection of disjoint subsets $ F $ with diameter $ h_F $, where $ F $ is an open subset of a hyperplane of $ \rd $ such that $ \snorm{F}_{d-1}>0 $ and $ \snorm{\overline{F}\backslash F}_{d-1}=0 $ with $ \snorm{\,\cdot\,}_{d-1} $ denoting the $ (d-1) $-dimensional Hausdorff measure. What's more, there holds that $ \bigcup_{T\in\mathcal{T}_h}\partial T = \bigcup_{F\in\mathcal{F}_h}\overline{F} $ and for all $ F\in\mathcal{F}_h $, (i) either $ F $ is an interface, i.e., there exist $ T_1, T_2\in\mathcal{T}_h $ such that $ F=\partial T_1 \cap\partial T_2 $, or (ii) $ F $ is a boundary face, i.e., there exists $ T\in\mathcal{T}_h $ such that $ F=\partial T \cap \partial \Omega $. We use $ \mathcal{F}_h^i,  \mathcal{F}_h^b $ to denote the collections of interfaces and boundary faces, respectively, such that $ \mathcal{F}_h = \mathcal{F}_h^i \cup \mathcal{F}_h^b $.
	
	For all $ T\in\mathcal{T}_h $, the set $ \mathcal{F}_T:=\{F\in\mathcal{F}_h:F\subset\partial T\} $ collects the faces lying on the boundary of $ T $ and, the sets $ \mathcal{T}_{N,T}:=\{T'\in\mathcal{T}_h:T'\cap T \neq \emptyset\} $, $ \mathcal{F}_{N,T}:=\{F\in\mathcal{F}_h:\overline{F}\cap \partial T \neq \emptyset\} $ collect the elements and faces sharing at least one node with $ T $. For all $ F\in\mathcal{F}_T $, we denote by $ \bm{n}_{TF} $ the normal to $ F $ pointing out of $ T $. For all $ F\in\mcalFh $, the set $ \mathcal{T}_F:=\{ T_1, T_2\in\mcalTh: \partial T_1\cap \partial T_2=F \} $ collects the elements whose boundary containing $ F $, and note that, for $ F\in\mcalFh^b $, we set $ T_1=T_2 $ to be compatible with this case. A normal vector $ \bm{n}_F $ is associated to interior faces by fixing once and for all an orientation, whereas for boundary faces $ \bm{n}_F $ points out of $ \Omega $.
	
	Additionally, we assume that the $ h $-refined mesh $ (\mathcal{T}_h)_{h\in\mathcal{H}} $ is regular in the sense of \citep[Definition 1.9]{2020_hho_book}, i.e., $ \mathcal{T}_h $ admits a matching simplicial submesh $ \mathfrak{T}_h $ and there exists a real number $ \varrho\in (0,1) $ independent of $ h $ such that for all $ h\in \mathcal{H} $, the following properties hold: (i) shape regularity: for all simplex $ S\in\mathfrak{T}_h $ with diameter $ h_S $ and inradius $ r_S $, $ \varrho h_S\leq r_S $; (ii) contact regularity: for all $ T\in\mathcal{T}_h $ and $ S\in\mathfrak{T}_h $ such that $ S\subset T $, $ \varrho h_T\leq h_S $. 
	
	Under the mesh regular settings above, for all $ h\in\mathcal{H} $, all $ T\in\mathcal{T}_h $, and all $ F\in\mcalF_T $, the following comparison of element and face diameters holds:
	\begin{align}\label{sec_model_eq4_1}
		2\varrho^2 h_T \leq h_F \leq h_T.
	\end{align}
	
	Let $ D $ be as in the previous section and denote by $ \mathbb{P}^l(D) $ $ (l\geq 0) $ the space spanned by the restrictions to $ D $ of polynomials in the space variables of total degree $ l $. We next introduce the $ L^2 $-orthogonal projector $ \pi_D^{0,l}: L^1(D)\rightarrow \mathbb{P}^l(D) $ such that, for all $ v\in L^1(D) $,
	\begin{align}\label{sec_model_eq5}
		(\pi_D^{0,l}v-v, w)_D = 0 \quad \forall w\in\mathbb{P}^l(D),
	\end{align}
	and the elliptic projector $ \pi_D^{1,l}: W^{1,1}(D)\rightarrow \mathbb{P}^l(D) $ such that, for all $ v\in W^{1,1}(D) $,
	\begin{align}\label{sec_model_eq6}
		(\nabla(\pi_D^{1,l}v-v), \nabla w)_D = 0 \quad \forall w\in\mathbb{P}^l(D),
	\end{align}
	with $ (\pi_D^{1,l}v-v,1)_D = 0 $. The vector (and matrix) valued  $ L^2 $-orthogonal and elliptic projectors, denoted by $ \bm{\pi}_D^{0,l} $ and $ \bm{\pi}_D^{1,l} $, respectively, are obtained applying $ \pi_D^{0,l} $ and $ \pi_D^{1,l} $ component-wise. To alleviate the notation, only scalar variables and operators are considered in the following representation, the vector (and matrix) cases can be obtained accordingly.
	
	The $ L^2 $-orthogonal and elliptic projectors have the following optimal $ W^{s,r} $-approximation properties (\citep[Chapter 1]{2020_hho_book}): let $ s\in \{0, ..., l+1\} $, $ r\in[1, +\infty] $, then there is a real number $ C>0 $ only depending on $ d $, $ \varrho $, $ l $, $ s $ and $ r $ such that, for all $ h\in\mathcal{H} $, $ D\in\mathcal{T}_h $ (or $ D\in\mathcal{F}_h $), and all $ v\in W^{s,r}(D) $, $ m\in\{0, ..., s-1\} $,
	\begin{align}\label{sec_model_eq7}
		\snorm{v - \pi_D^{0,l}v}_{m,r,D} \leq C h_D^{s-m}\snorm{v}_{s,r,D}, \qquad \snorm{v - \pi_D^{1,l}v}_{m,r,D} \leq C h_D^{s-m}\snorm{v}_{s,r,D},
	\end{align}
	and for all $ T\in\mathcal{T}_h $, $ F\in\mathcal{F}_T $, it holds that 
	\begin{align}\label{sec_model_eq8}
		h_T^{\frac{1}{r}}\snorm{v-\pi_T^{1,l}v}_{m,r,F} \leq C h_T^{s-m}\snorm{v}_{s,r,T},
	\end{align}
	moreover, if $ s\geq 1 $, 
	\begin{align}\label{sec_model_eq9}
		h_T^{\frac{1}{r}}\snorm{v-\pi_T^{0,l}v}_{m,r,F} \leq C h_T^{s-m}\snorm{v}_{s,r,T}.
	\end{align}
	We also need the following discrete inverse inequality: for all $ T\in\mathcal{T}_h $ and $ v\in\mathbb{P}^l(T) $, there exists real number $ C>0 $ independent of $ h $ and of $ T $ such that,
	\begin{align}\label{sec_model_eq10}
		\dnorm{\nabla v}_{0,r,T} \leq C h_T^{-1}\dnorm{v}_{0,r,T}.
	\end{align}
	And we recall the continuous trace inequality: there is a real number $ C>0 $ independent of $ h $ such that, for all $ T\in\mathcal{T}_h $ it holds for all $ v\in W^{1,r}(T) $,
	\begin{align}\label{sec_model_eq11}
		\sum_{F\in\mcalF_T} h_T^{\frac{1}{r}}\dnorm{v}_{0,r,F} \leq C (\dnorm{v}_{0,r,T}+h_T\dnorm{\nabla v}_{0,r,T}).
	\end{align}
	By combining \eqref{sec_model_eq10} and \eqref{sec_model_eq11}, we get the discrete the trace inequality: for all $ T\in\mathcal{T}_h $ and $ v\in \mathbb{P}^l(T) $, it holds that
	\begin{align}\label{sec_model_eq12}
		\sum_{F\in\mcalF_T} h_T^{\frac{1}{r}}\dnorm{v}_{0,r,F} \leq C\dnorm{v}_{0,r,T},
	\end{align}
	where $ C>0 $ is a real number independent of $ h $.
	
	For completeness, at the global level, we define the broken polynomial space 
	\begin{align*}
		\mathbb{P}^l(\mathcal{T}_h) := \Big\{v\in L^2(\Omega): v_{|T}\in\mathbb{P}^l(T)\hspace{1em} \forall T\in\mathcal{T}_h\Big\},
	\end{align*}
	the broken Sobolev space 
	\begin{align*}
		W^{s,r}(\mathcal{T}_h) := \Big\{v\in L^r(\Omega): v_{|T}\in W^{s,r}(T) \hspace{1em} \forall T\in\mathcal{T}_h \Big\},
	\end{align*}
	on which, for all $ \bm{v}_h\in\mathbb{P}^l(\mathcal{T}_h) $ (or $ \bm{v}_h\in W^{s,r}(\mathcal{T}_h) $), the norm and seminorm are defined by 
	\begin{align*}
		\dnorm{\bm{v}_h}_{\star, \mcalTh} := \sum_{T\in\mcalTh} \dnorm{\bm{v}_h}_{\star, T}, \quad \snorm{\bm{v}_h}_{\star, \mcalTh} := \sum_{T\in\mcalTh} \snorm{\bm{v}_h}_{\star, T},
	\end{align*}
	where $ \star $ may take different values in different places. The broken gradient operator {defined on $ W^{1,1}(\mathcal{T}_h) $} is denoted by $ \nabla_h $, similarly, the $ L^2 $-orthogonal and elliptic projectors {defined on $ \mathbb{P}^l(\mathcal{T}_h) $} are respectively denoted by $ \pi_h^{0,l} $ and $ \pi_h^{1,l} $.

	\section{Discretization}\label{sec_discretization}
	In this section, we define the spaces of discrete unknowns and the local reconstructions, state the discrete problem. 
	
	\subsection{Discrete spaces}
	Let a polynomial degree $ k\geq 0 $ be fixed. We define the discrete space for the velocity as 
	\begin{align*}
		\hhoVecvars{V}_h^k:=\Big\{ \hhoVecvars{v}_h=((\bm{v}_T)_{T\in\mathcal{T}_h}, (\bm{v}_F)_{F\in\mathcal{F}_h}): \bm{v}_T\in\sqd{\mbbP^k(T)} \ \ \forall T\in\mcalT_h \text{ \ and \ } \bm{v}_F\in\sqd{\mbbP^k(F)} \ \ \forall F\in\mcalF_h \Big\}.
	\end{align*}
	To account for the homogeneous Dirichlet boundary condition in a stronger manner for the velocity, we introduce the subspace
	\begin{align*}
		\hhoVecvars{V}_{h,0}^k:=\Big\{ \hhoVecvars{v}_h\in\hhoVecvars{V}_{h}^k: \bm{v}_F=\bm{0} \quad \forall F\in\mathcal{F}_h^b \Big\},
	\end{align*}
	and we define the zero-average constraint space for the pressure as follows
	\begin{align*}
		Q_h^k:=\Big\{ q_h\in\mbbP^k(\mcalTh): \int_\Omega q_h = 0 \Big\}.
	\end{align*}
	The restriction of $ \hhoVecvars{v}_h\in\hhoVecvars{V}_h^k $ and $ q_h\in Q_h^k $ to an element $ T\in\mcalT_h $ is denoted by $ \hhoVecvars{V}_T^k $, $ \hhoVecvars{v}_T=(\bm{v}_T, (\bm{v}_F)_{F\in\mcalF_T}) $ and $ q_T $, respectively. Also, we denote by $ \bm{v}_h $ (no underline) the function in $ \sqd{\mbbP^k(\mcalT_h)} $ such that 
	\begin{align*}
		\bm{v}_{h|T}:=\bm{v}_T \quad \forall T\in\mcalT_h,
	\end{align*}
	and for any $ F\in\mcalFh^i $, denote by $ T_1 $, $ T_2 $ the distinct elements of $ \mcalTh $ such that $ F\subset \partial T_1 \cap \partial T_2 $, moreover, we fix an arbitrary numbering of $ T_1 $ and $ T_2 $ to introduce the jump of $ \bm{v}_h $ across $ F\in\mcalFh $, as 
	\begin{equation}\label{sec_model_eq13_1}
		\jump{\bm{v}_h}_{F} := \left\{
		\begin{aligned}
			&(\bm{v}_{h| T_1} )_{|F} - (\bm{v}_{h| T_2} )_{|F}   \quad \text{if } F\in\mcalFh^i,\\
			& (\bm{v}_h)_{|F} \hspace{7.21em} \text{if } F\in\mcalFh^b.
		\end{aligned}
		\right.
	\end{equation}
	
	We define on $ \hhoVecvars{V}_h^k $ the seminorm $ \dnorm{\,\cdot\,}_{1,h} $ such that, here, we again abuse the notation $ \dnorm{\,\cdot\,}_{1,T} $, for all $ \hhoVecvars{v}_h\in\hhoVecvars{V}_h^k $,
	\begin{align}\label{sec_model_eq13}
		\dnorm{\hhoVecvars{v}_h}^2_{1,h}:=\sum_{T\in\mcalT_h}\dnorm{\hhoVecvars{v}_T}_{1,T}^2,
	\end{align}
	where, for all $ T\in\mcalT_h $,
	\begin{align}\label{sec_model_eq14}
		\dnorm{\hhoVecvars{v}_T}_{1,T}^2:= \dnorm{\nabla\bm{v}_T}_T^2 + \snorm{\hhoVecvars{v}_T}_{1,\partial T}^2, \qquad \snorm{\hhoVecvars{v}_T}_{1,\partial T}^2:=\sum_{F\in\mcalF_h}h_F^{-1}\dnorm{\bm{v}_F-\bm{v}_T}_F^2.
	\end{align}
	Indeed, it is easy to verify that the map $ \dnorm{\,\cdot\,}_{1,h} $ defines a norm on $ \hhoVecvars{V}_{h,0}^k $.

	The vector of discrete variables corresponding to a smooth function on $ \Omega $ is obtained by the global interpolation operator $ \hhoVecvars{I}_h^k: \sqd{H^1(\Omega)}\rightarrow \hhoVecvars{V}_h^k $ such that, for all $ \bm{v}\in\sqd{H^1(\Omega)} $
	\begin{align}\label{sec_model_eq15}
		\hhoVecvars{I}_h^k\bm{v}:=((\bm{\pi}_T^{0,k}\bm{v}_{|T})_{T\in\mcalT_h}, (\bm{\pi}_F^{0,k}\bm{v}_{|F})_{F\in\mcalF_h}),
	\end{align}
	its restriction to an element $ T\in\mcalT_h $ is denoted by $ \hhoVecvars{I}_T^k $. The following boundedness property holds for the local interpolation operator $ \hhoVecvars{I}_T^k $ (\citep[Proposition 2.2]{2020_hho_book}): there exists a real number $ C_I>0 $ independent of $ h $ and $ \nu $, but possibly depending on $ d $, $ \varrho $ and $ k $, such that, for all $ \bm{v}\in\sqd{H^1(T)} $,
	\begin{align}\label{sec_model_eq16}
		\dnorm{\hhoVecvars{I}_T^k\bm{v}}_{1,T} \leq C_I \dnorm{\nabla\bm{v}}_T.
	\end{align}
	With the above boundedness of $ \hhoVecvars{I}_T^k $, we have 
	\begin{align}\label{sec_model_eq16_1}
		h_F^{-\frac{1}{2}}\dnorm{v-\pi_F^{0,l}v}_{F}\leq C\dnorm{\nabla v}_{T},
	\end{align}
	in which, $ C>0 $ is a real number independent of $ h $ and, the detailed proof can be found in \citep[Proposition 4.6]{2020_hho_book}.

	\subsection{Local reconstructions}
	Let an element $ T\in\mcalTh $ and polynomial degree $ k\geq 0 $ be fixed. The local velocity reconstruction operator $ \bm{r}_T^{k+1}: \hhoVecvars{V}_T^k\rightarrow\sqd{\mbbP^{k+1}(T)} $ is defined such that, for all $ \hhoVecvars{v}_T\in\hhoVecvars{V}_T^k $ and $ \matvar{\tau}\in\nabla\sqd{\mbbP^{k+1}(T)} $,
	\begin{align}\label{sec_model_eq17}
		(\nabla \bm{r}_T^{k+1}\hhoVecvars{v}_T, \matvar{\tau})_T = (\nabla \bm{v}_T, \matvar{\tau})_T + \sum_{F\in\mcalF_T}(\bm{v}_F-\bm{v}_T, \matvar{\tau}\bm{n}_{TF})_F,
	\end{align}
	and the mean-value of $ \bm{r}_T^{k+1}\hhoVecvars{v}_T $ in $ T $ is set equal to that of $ \bm{v}_T $. The divergence reconstruction operator $ D_T^k:\hhoVecvars{V}_T^k\rightarrow \mbbP^k(T) $ is defined such that, for all $ \hhoVecvars{v}_T\in\hhoVecvars{V}_T^k $ and $ q\in\mbbP^k(T) $,
	\begin{align}\label{sec_model_eq18}
		(D_T^k\hhoVecvars{v}_T, q)_T = (\nabla\cdot\bm{v}_T, q)_T + \sum_{F\in\mcalF_T}(\bm{v}_F-\bm{v}_T,q\bm{n}_{TF})_F.
	\end{align}
	By recalling the definition of $ \pi_T^{0,k} $, $ \bm{\pi}_T^{1,k} $ and $ \hhoVecvars{I}_T^k $, we infer that, for all $ \bm{v}\in\sqd{H^1(T)} $,
	\begin{align}\label{sec_model_eq19}
		\bm{r}_T^{k+1}\hhoVecvars{I}_T^k\bm{v} = \bm{\pi}_T^{1,k+1}\bm{v}, \qquad D_T^k\hhoVecvars{I}_T^k\bm{v} = \pi_T^{0,k}(\nabla\cdot \bm{v}).
	\end{align}
	
	We also define the operators $ \bdeltaT: \hhoVecvars{V}_T^k\rightarrow \sqd{\mbbP^k(T)} $ and, for all $ F\in\mcalF_T $, $ \bdeltaF: \hhoVecvars{V}_T^k\rightarrow \sqd{\mbbP^k(F)} $ such that, for all $ \hhoVecvars{v}_T\in\hhoVecvars{V}_T^k $,
	\begin{align}\label{sec_model_eq19_1}
		\bdeltaT\hhoVecvars{v}_T:=\bm{\pi}_T^{0,k}(\bm{r}_T^{k+1}\hhoVecvars{v}_T-\bm{v}_T), \quad \bdeltaF\hhoVecvars{v}_T:= \bm{\pi}_F^{0,k}(\bm{r}_T^{k+1}\hhoVecvars{v}_T-\bm{v}_F)\hspace{1em} \forall F\in\mcalF_T, 
	\end{align}
	which satisfies 
	\begin{align}\label{sec_model_eq19_2}
		(\bdeltaT\hhoVecvars{v}_T, (\bdeltaF\hhoVecvars{v}_T)_{F\in\mcalF_T}) = \hhoVecvars{I}_T^k\bm{r}_T^{k+1}\hhoVecvars{v}_T - \hhoVecvars{v}_T.
	\end{align}
	
	\subsection{Viscous term}
	The viscous term is discretised by means of the bilinear form $ a_h:\hhoVecvars{V}_h^k\times\hhoVecvars{V}_h^k\rightarrow \realR $ such that, for all $ \hhoVecvars{w}_h, \hhoVecvars{v}_h\in \hhoVecvars{V}_h^k $,
	\begin{align}\label{sec_model_eq20}
		a_h(\hhoVecvars{w}_h, \hhoVecvars{v}_h):=\sum_{T\in\mcalT_h} a_T(\hhoVecvars{w}_T,\hhoVecvars{v}_T),
	\end{align}
	where, the local contribution is such that 
	\begin{align}\label{sec_model_eq21}
		a_T(\hhoVecvars{w}_T,\hhoVecvars{v}_T) := (\nabla\bm{r}_T^{k+1}\hhoVecvars{w}_T,\nabla\bm{r}_T^{k+1}\hhoVecvars{v}_T )_T + s_T(\hhoVecvars{w}_T,\hhoVecvars{v}_T),
	\end{align}
	with the stabilization bilinear form $ s_T $ defined as 
	\begin{align}\label{sec_model_eq21_1}
		s_T(\hhoVecvars{w}_T, \hhoVecvars{v}_T) := h_T^{-2} (\bdeltaT\hhoVecvars{w}_T, \bdeltaT\hhoVecvars{v}_T)_T +  \sum_{F\in\mcalF_T}h_F^{-1}(\bdeltaF\hhoVecvars{w}_T, \bdeltaF\hhoVecvars{v}_T)_F,
	\end{align}
	
	The subtle choice of $ s_T $ ensures the following designed conditions have been originally presented in \cite{2018_ESAIM_Daniele}.
	\begin{Lemma}[Local stabilisation bilinear form $ s_T $] \label{sec_model_lem1}
		The proposed local stabilisation bilinear form $ s_T $ in \eqref{sec_model_eq21_1} satisfies the following properties:
		\begin{enumerate}
			\item[\rm{(S1)}] {\rm Symmetry and positivity.} $ s_T $ is symmetric and positive semidefinite;
			\item[\rm{(S2)}] {\rm Stability and boundedness.} There is a real number $ \lambda $ independent of $ h $ and $ T $ such that, for all $ \hhoVecvars{v}_T\in\hhoVecvars{V}_T^k $,
			\begin{align}\label{sec_model_eq22}
				\lambda\dnorm{\hhoVecvars{v}_T}_{1,T}^2\leq a_T(\hhoVecvars{v}_T,\hhoVecvars{v}_T) \leq \lambda^{-1} \dnorm{\hhoVecvars{v}_T}_{1,T}^2;
			\end{align}
			\item[\rm{(S3)}] {\rm Polynomial consistency.} For all $ \bm{w}\in\sqd{\mbbP^{k+1}(T)} $ and $ \hhoVecvars{v}_T\in\hhoVecvars{V}_T^k $, it holds
			\begin{align}\label{sec_model_eq22_01}
				s_T(\hhoVecvars{I}_T^k\bm{w},\hhoVecvars{v}_T) = 0.
			\end{align}
		\end{enumerate}
	\end{Lemma}
	\begin{proof}
		Clearly, from the definition of $ s_T $, the property (S1) holds, and the property (S3) is a consequence of \citep[Lemma 2.11]{2020_hho_book}. It only remains to prove property (S2). In what follows of the rest proof, we let $ \hhoVecvars{v}_T $ be a generic element of $ \hhoVecvars{V}_T^k $, and denote by 
		\begin{align*}
			\bm{\check{v}}_T := \bm{r}_T^{k+1} \hhoVecvars{v}_T.
		\end{align*}
		
		We first directly give two bounds about $ \dnorm{\nabla \bm{v}_T}_T $ and $ \dnorm{\nabla\bm{\check{v}}_T}_T $, which have been proved in \citep[Proposition 2.13]{2020_hho_book}, as follows
		\begin{subequations}\label{sec_model_eq22_02}
			\begin{align}
				\label{sec_model_eq22_02a}
				\dnorm{\nabla \bm{v}_T}_T &\lesssim \dnorm{\nabla\bm{\check{v}}_T}_T + \snorm{\hhoVecvars{v}_T}_{1,\partial T}, \\
				\label{sec_model_eq22_02b}
				\dnorm{\nabla\bm{\check{v}}_T}_T &\lesssim \dnorm{\hhoVecvars{v}_T}_{1,T},
			\end{align}
		\end{subequations}
		where, as well as for the rest of this paper, we use the abbreviation $ a\lesssim b $ for the inequality $ a\leq Cb $ with generic positive constant $ C $ independent of mesh size $ h $ and viscosity $ \nu $. 
		
		And, invoking the boundedness \eqref{sec_model_eq16} of $ \hhoVecvars{I}_T^k $ with $ \bm{v}=\bm{\check{v}}_T $ to get 
		\begin{align}\label{sec_model_eq22_02_1}
			\snorm{\hhoVecvars{I}_T^k\bm{\check{v}}_T}_{1,\partial T} \leq \dnorm{\hhoVecvars{I}_T^k\bm{\check{v}}_T}_{1, T} \leq \dnorm{\nabla \bm{\check{v}}_T}_T.
		\end{align}
		
		Then, in the ensuing proof, we prove the property (S2). Using the definition \eqref{sec_model_eq21_1} of $ a_T $, we have 
		\begin{align}\label{sec_model_eq22_04}
			a_T(\hhoVecvars{v}_T,\hhoVecvars{v}_T) = \dnorm{\nabla \bm{\check{v}}_T}_T^2 + h_T^{-2}\dnorm{\bdeltaT\hhoVecvars{v}_T}_T^2 + \sum_{F\in\mcalF_T}h_T^{-1}\dnorm{\bdeltaF\hhoVecvars{v}_T}_F^2.
		\end{align}
		By using the relation \eqref{sec_model_eq4_1} and discrete trace inequality \eqref{sec_model_eq12}, it holds
		\begin{align}\label{sec_model_eq22_05}
			\sum_{F\in\mcalF_T}h_F^{-1}\dnorm{\bdeltaT\hhoVecvars{v}_T}_F^2 &\leq \sum_{F\in\mcalF_T} \frac{1}{2\varrho^2}h_T^{-1}\dnorm{\bdeltaT\hhoVecvars{v}_T}_F^2 \notag \\
			&\leq \sum_{F\in\mcalF_T} \frac{1}{2\varrho^2}h_T^{-2}\dnorm{\bdeltaT\hhoVecvars{v}_T}_T^2 \notag \\
			& \leq N_{\partial T} \frac{1}{2\varrho^2}h_T^{-2}\dnorm{\bdeltaT\hhoVecvars{v}_T}_T^2,
		\end{align}
		where $ N_{\partial T} $ denotes the number of faces in $ T $. 
		
		Let
		\begin{align}\label{sec_model_eq22_03}
			\hhoVecvars{y}_T := \hhoVecvars{I}_T^k\bm{\check{v}}_T - \hhoVecvars{v}_T = (\bdeltaT\hhoVecvars{v}_T, (\bdeltaF\hhoVecvars{v}_T)_{F\in\mcalF_T}).
		\end{align}
		On the one hand, from \eqref{sec_model_eq22_05} and note that $ \varrho\in (0,1) $, we get
		\begin{align}\label{sec_model_eq22_06}
			\snorm{\hhoVecvars{y}_T}_{1,\partial T}^2 &= \sum_{F\in\mcalF_T} h_F^{-1}\dnorm{\bdeltaF\hhoVecvars{v}_T - \bdeltaT\hhoVecvars{v}_T}_F^2 \notag \\
			&\leq  2\sum_{F\in\mcalF_T} h_F^{-1}\dnorm{\bdeltaF\hhoVecvars{v}_T}_F^2 + 2\sum_{F\in\mcalF_T} h_F^{-1}\dnorm{\bdeltaT\hhoVecvars{v}_T}_F^2 \notag \\
			&\leq N_{\partial T}\frac{2}{\varrho^2}\Big(\sum_{F\in\mcalF_T}h_F^{-1}\dnorm{\bdeltaF\hhoVecvars{v}_T}_F^2 + h_T^{-2}\dnorm{\bdeltaT \hhoVecvars{v}_T}_T^2 \Big).
		\end{align}
		On the other hand, collecting the above inequalities \eqref{sec_model_eq22_02_1} and \eqref{sec_model_eq22_03}$ - $\eqref{sec_model_eq22_06} leads to
		\begin{align*}
			\snorm{\hhoVecvars{v}_T}_{1,\partial T}^2 &= \snorm{\hhoVecvars{I}_T^k\bm{\check{v}}_T - \hhoVecvars{y}_T}_{1,\partial T}^2 \\
			&\leq 2\snorm{\hhoVecvars{I}_T^k\bm{\check{v}}_T}_{1,\partial T}^2 + 2\snorm{ \hhoVecvars{y}_T}_{1,\partial T}^2 \\
			& \lesssim \dnorm{\nabla \bm{\check{v}}_T}^2 + h_T^{-2}\dnorm{\bdeltaT\hhoVecvars{v}_T}_T^2 + \sum_{F\in\mcalF_T}h_F^{-1}\dnorm{\bdeltaF\hhoVecvars{v}_T}_F^2 = a_T(\hhoVecvars{v}_T,\hhoVecvars{v}_T).
		\end{align*}
		Combining this estimate and \eqref{sec_model_eq22_02a} yields
		\begin{align*}
			\dnorm{\hhoVecvars{v}_T}_{1,T}^2 = \dnorm{\nabla \bm{v}_T}_T^2 + \snorm{\hhoVecvars{v}_T}_{1,\partial T}^2 \lesssim \dnorm{\bm{\check{v}}_T}_T^2 + \snorm{\hhoVecvars{v}_T}_{1,\partial T}^2 \lesssim a_T(\hhoVecvars{v}_T,\hhoVecvars{v}_T),
		\end{align*}
		which gives the first estimate in \eqref{sec_model_eq22}, and the second will be proved in the following.
		
		We start from the estimation of $ \snorm{\hhoVecvars{y}_T}_{1,\partial T}^2 $,
		\begin{align}\label{sec_model_eq22_07}
			\snorm{\hhoVecvars{y}_T}_{1,\partial T}^2 \leq 2\snorm{\hhoVecvars{I}_T^k\bm{\check{v}}_T}_{1,\partial T}^2 + 2 \snorm{\hhoVecvars{v}_T}_{1,\partial T}^2 \lesssim \dnorm{\nabla \bm{\check{v}}_T}_T^2 + \snorm{\hhoVecvars{v}_T}_{1,\partial T}^2.
		\end{align}
		From the definition \eqref{sec_model_eq17} of $ \bm{r}_T^{k+1} $, we know that $ \int_T\bm{\pi}_T^{0,k}(\bm{\check{v}}_T-\bm{v}_T) = \int_T\bm{\check{v}}_T-\bm{v}_T = 0 $, then by using the local Poincar\'{e}-Wirtinger inequality and \eqref{sec_model_eq22_02a} leads to
		\begin{align}\label{sec_model_eq22_08}
			h_T^{-2}\dnorm{\bdeltaT\hhoVecvars{v}_T}_T^2 \leq h_T^{-2}(h_T^2\dnorm{ \nabla(\bm{\check{v}}_T - \bm{v}_T) }_T^2) \leq 2\dnorm{ \nabla\bm{\check{v}}_T }_T^2 + 2\dnorm{\nabla\bm{v}_T }_T^2 \lesssim \dnorm{\nabla\bm{\check{v}}_T}_T + \snorm{\hhoVecvars{v}_T}_{1,\partial T}.
		\end{align}
		Consequently, using the triangle inequality in \eqref{sec_model_eq22_04} and then plugging \eqref{sec_model_eq22_02b}, \eqref{sec_model_eq22_05} and \eqref{sec_model_eq22_07}$ - $\eqref{sec_model_eq22_08} into it gives
		\begin{align*}
			a_T(\hhoVecvars{v}_T,\hhoVecvars{v}_T) &\leq \dnorm{\nabla \bm{\check{v}}_T}_T^2 + h_T^{-2}\dnorm{\bdeltaT\hhoVecvars{v}_T}_T^2 + 2\sum_{F\in\mcalF_T}h_T^{-1}\dnorm{\bdeltaF\hhoVecvars{v}_T-\bdeltaT\hhoVecvars{v}_T}_F^2 + 2\sum_{F\in\mcalF_T}h_T^{-1}\dnorm{\bdeltaT\hhoVecvars{v}_T}_F^2\\
			&\lesssim \dnorm{\nabla \bm{\check{v}}_T}_T^2 + h_T^{-2}\dnorm{\bdeltaT\hhoVecvars{v}_T}_T^2 + \snorm{\hhoVecvars{y}_T}_{1,\partial T}^2 \\
			&\lesssim \dnorm{\nabla \bm{\check{v}}_T}_T^2 + \snorm{\hhoVecvars{v}_T}_{1,\partial T}^2 \\
			&\lesssim \dnorm{\hhoVecvars{v}_T}_{1,T}^2,
		\end{align*}
		which completes the proof of (S2).
		\qedhere
	\end{proof}

	Besides, for the bilinear form $ s_T $ satisfies (S1)$ - $(S3) has the following property.
	\begin{Lemma}[{\citep[Proposition 2.14]{2020_hho_book}}]
		Let $ T\in\mcalTh $, $ r\in\{-1, 0, ..., k\} $. For all $ \bm{v}\in\sqd{H^{r+2}(T)} $, 
		\begin{align}\label{sec_model_eq22_1}
			s_T(\hhoVecvars{I}_h^k\bm{v},\hhoVecvars{I}_h^k\bm{v})^{\frac{1}{2}}\lesssim h_T^{r+1}\snorm{\bm{v}}_{r+2,T}.
		\end{align}
	\end{Lemma}

	\subsection{Pressure-velocity coupling}
	The pressure-velocity coupling hings on the bilinear form $ b_h $ on $ \hhoVecvars{V}_h^k\times \mbbP^k(\mcalTh) $ such that,
	\begin{align}\label{sec_model_eq23}
		b_h(\hhoVecvars{v}_h,q_h) := -\sum_{T\in\mcalT_h} (D_T^k\hhoVecvars{v}_T, q_T).
	\end{align}
	This bilinear form $ b_h $ enjoys the discrete inf-sup stability: there is a real number $ \widetilde{\beta} $ independent of $ h $, such that
	\begin{align}\label{sec_model_eq24}
		\widetilde{\beta}\leq \inf_{q_h\in Q_h\backslash\{0\}}\sup_{\hhoVecvars{v}_h\in\hhoVecvars{V}_{h,0}^k\backslash\{\hhoVecvars{0}\}} \frac{b_h(\hhoVecvars{v}_h, q_h)}{\dnorm{\hhoVecvars{v}_h}_{1,h}\dnorm{q_h}}.
	\end{align}
	
	\subsection{Discrete problem}
	The HHO scheme for the approximation of \eqref{sec_model_eq2} can be obtained by finding $ (\hhoVecvars{u}_h, p_h)\in \hhoVecvars{V}_{h,0}^k\times Q_h^k $ satisfying the following equations:
	\begin{subequations}\label{sec_model_eq25}
		\begin{align}
			\label{sec_model_eq25a}
			\nu a_h(\hhoVecvars{u}_h, \hhoVecvars{v}_h) + b_h(\hhoVecvars{v}_h, p_h) &= (\bm{f}, \bm{v}_h) \quad \forall \hhoVecvars{v}_h\in\hhoVecvars{V}_{h,0}^k, \\
			\label{sec_model_eq25b}
			-b_h(\hhoVecvars{u}_h, q_h) &= 0 \hspace{3.5em} \forall q_h \in Q_h^k. 
		\end{align}
	\end{subequations}
	Combining \eqref{sec_model_eq22} and \eqref{sec_model_eq24} with the saddle point theorem, we know that the discrete problem \eqref{sec_model_eq25} is well-posed. In addition, define the global reconstruction operator $ \bm{r}_h^{k+1}: \hhoVecvars{V}_h^k\rightarrow\sqd{\mbbP^{k+1}(\mcalTh)} $ such that, for all $ \hhoVecvars{v}_h\in\hhoVecvars{V}_h^k $, 
	\begin{align*}
		(\bm{r}_h^{k+1}\hhoVecvars{v}_h)_{|T} := \bm{r}_T^{k+1}\hhoVecvars{v}_T  \quad \forall T\in\mcalTh,
	\end{align*}
	also, for the sake of brevity, denote by 
	\begin{align*}
		s_h(\hhoVecvars{v}_h,\hhoVecvars{v}_h) := \sum_{T\in\mcalTh} s_T(\hhoVecvars{v}_T,\hhoVecvars{v}_T).
	\end{align*}

	Along with the reference \citep[Theorem 8.18, Theorem 8.20]{2020_hho_book}, the following lemma, for the a priori error estimate, holds true,
	\begin{Lemma}\label{sec_model_lem2}
		Let polynomial degree $ k\geq 0 $ be fixed and $ (\bm{u}, p)\in\bm{V}\times Q $ denote the unique solution to the continuous problem \eqref{sec_model_eq2}, for which we assume the additional regularity $ \bm{u}\in \sqd{H^{r+2}(\mcalTh)} $ and $ p\in H^1(\Omega)\cap H^{r+1}(\mcalTh) $ for some $ r\in\{0,...,k\} $. For all $ h\in\mathcal{H} $, let $ (\hhoVecvars{u}_h, p_h)\in\hhoVecvars{V}_{h,0}^k\times Q_h^k $ denote the unique solution to the discrete problem \eqref{sec_model_eq25}. Define the discretization errors $ \bm{e_u} $ and $ e_p $ as follows
		\begin{subequations}\label{sec_model_eq26_1}
			\begin{align}
				&\bm{e_u}^2 := \sum_{T\in\mcalT_h} \bm{e}_{\bm{u},T}^2, \quad  \bm{e}_{\bm{u},T}^2:= \nu\dnorm{\nabla\bm{u}-\nabla\bm{r}_T^{k+1}\hhoVecvars{u}_T}_T^2 + \nu s_T(\hhoVecvars{u}_T,\hhoVecvars{u}_T), \\
				&e_p^2 := \sum_{T\in\mcalT_h} e_{p,T}^2, \quad e_{p,T}^2:= \nu^{-1}\dnorm{p - p_T}_T^2.
			\end{align}
		\end{subequations}
		Then, 
		\begin{align}\label{sec_model_eq27}
			\big(\bm{e_u}^2 + e_p^2\big)^{\frac{1}{2}} \lesssim h^{r+1}\big( \nu^{\frac{1}{2}}\snorm{\bm{u}}_{r+2,\mcalTh} + \nu^{-\frac{1}{2}}\snorm{p}_{r+1,\mcalTh} \big).
		\end{align}
	\end{Lemma}
	\begin{proof}
		We recall the conclusion of \citep[Theorem 8.18]{2020_hho_book}:
		\begin{align}\label{sec_model_eq28}
			\nu\dnorm{\hhoVecvars{u}_h-\hhoVecvars{I}_h^k\bm{u}}_{1,h}^2 + \nu^{-1}\dnorm{p_h-\pi_h^{0,k}p}^2 \lesssim h^{r+1}\big( \nu^{\frac{1}{2}}\snorm{\bm{u}}_{r+2,\mcalTh} + \nu^{-\frac{1}{2}}\snorm{p}_{r+1,\mcalTh} \big),
		\end{align}
		and note that the proof of the estimation \eqref{sec_model_eq28} holds for all $ s_T $ satisfying the properties in Lemma \ref{sec_model_lem1}.
		
		To estimate the error $ \bm{e_u} $, we insert $ \pm\nabla_h\bm{r}_h^{k+1}\hhoVecvars{I}_h^k\bm{u} $ into the first term, $ \pm \hhoVecvars{I}_h^k\bm{u} $ into the second, and using the consistency (S3) of $ s_T $ and the triangle inequality, we get
		\begin{align*}
			&\dnorm{\nabla \bm{u} - \nabla_h\bm{r}_h^{k+1}\hhoVecvars{u}_h} + s_h(\hhoVecvars{u}_h,\hhoVecvars{u}_h)^{\frac{1}{2}}\\
			& \leq \Big( \dnorm{\nabla_h\bm{r}_h^{k+1}(\hhoVecvars{u}_h-\hhoVecvars{I}_h^k\bm{u})} + s_h(\hhoVecvars{u}_h-\hhoVecvars{I}_h^k\bm{u},\hhoVecvars{u}_h-\hhoVecvars{I}_h^k\bm{u})^{\frac{1}{2}} \Big) + \Big( \dnorm{\nabla_h (\bm{r}_h^{k+1}\hhoVecvars{I}_h^k\bm{u}-\bm{u})} + s_h(\hhoVecvars{I}_h^k\bm{u},\hhoVecvars{I}_h^k\bm{u})^{\frac{1}{2}} \Big)\\
			&=: \mathfrak{T}_1 + \mathfrak{T}_2.
		\end{align*}
		Along with the stability and boundedness (S2) of $ a_T $, by adding all the elements together, we have 
		\begin{align*}
			\nu^{\frac{1}{2}}\mathfrak{T}_1\lesssim \nu^{\frac{1}{2}}\dnorm{\hhoVecvars{u}_h-\hhoVecvars{I}_h^k\bm{u}}_{1,h},
		\end{align*}
		on the other hand, using \eqref{sec_model_eq19} and \eqref{sec_model_eq22_1} yields
		\begin{align*}
			\nu^{\frac{1}{2}}\mathfrak{T}_2\lesssim \nu^{\frac{1}{2}} h^{r+2}\snorm{\bm{u}}_{r+2,\mcalTh}.
		\end{align*}
		Then, from the triangle inequality and the property of projection \eqref{sec_model_eq7}, we obtain
		\begin{align*}
			\nu^{-\frac{1}{2}}\dnorm{p - p_h} &\leq \nu^{-\frac{1}{2}}\dnorm{p - \pi_h^{0,k}p} + \nu^{-\frac{1}{2}}\dnorm{p_h - \pi_h^{0,k}p} \\ 
			&\leq h^{r+1}\nu^{-\frac{1}{2}}\snorm{p}_{r+1,\mcalTh} + \nu^{-\frac{1}{2}}\dnorm{p_h - \pi_h^{0,k}p}.
		\end{align*}
		
		The result then follows by combining the individual bounds above with \eqref{sec_model_eq28}.
		\qedhere
	\end{proof}

	\section{A Posteriori Error Analysis for HHO}\label{sec_posterioriError}
	In this section, an a residual-type posteriori error estimator will be presented and analyzed for the discrete scheme \eqref{sec_model_eq25}.
	
	We first introduce the $ \bcurl $ operator for the vector functions $ \bm{v}=(v_1, \cdots, v_d)^\transpose $ (where $ \transpose $ denotes the transpose), such that, for $ d=2 $, 
	\begin{align*}
		\bcurl \bm{v} = \begin{pmatrix}
			-\partial_{x_2}v_1 \  \  \partial_{x_1}v_1 \\
			-\partial_{x_2}v_2 \  \  \partial_{x_1}v_2
		\end{pmatrix},
	\end{align*}
	and for $ d=3 $,
	\begin{align*}
		\bcurl \bm{v} = \begin{pmatrix}
			\partial_{x_2}v_1  - \partial_{x_3}v_1  \  \  \partial_{x_3}v_1  - \partial_{x_1}v_1  \ \ \partial_{x_1}v_1  - \partial_{x_2}v_1  \\
			\partial_{x_2}v_2  - \partial_{x_3}v_2  \  \  \partial_{x_3}v_2  - \partial_{x_1}v_2  \ \ \partial_{x_1}v_2  - \partial_{x_2}v_2  \\
			\partial_{x_2}v_3  - \partial_{x_3}v_3  \  \  \partial_{x_3}v_3  - \partial_{x_1}v_3  \ \ \partial_{x_1}v_3  - \partial_{x_2}v_3 
		\end{pmatrix}.
	\end{align*}
	
	Given an approximation $ (\hhoVecvars{u}_h, p_h)\in \hhoVecvars{V}_h^k\times Q_h^k $, we introduce the a posteriori error estimator $ \eta_h $ as follows,
	\begin{align}\label{sec_post_eq2}
		\eta_h^2 := \eta_{d,h}^2 + \eta_{s,h}^2 + \eta_{J,h}^2,
	\end{align}
	where
	\begin{align*}
		\eta_{d,h}^2 &:= \sum_{T\in\mcalT_h}\eta_{d,T}^2, \quad  \eta_{d,T}^2:= \nu \dnorm{\nabla\cdot\bm{r}_T^{k+1}\hhoVecvars{u}_T}_T^2, \\
		\eta_{s,h}^2 &:= \sum_{T\in\mcalT_h}\eta_{s,T}^2, \quad \eta_{s,T}^2:= \nu s_T(\hhoVecvars{u}_T, \hhoVecvars{u}_T), \\
		\eta_{J,h}^2 &:= \sum_{T\in\mcalT_h}\eta_{J,T}^2, \quad  \eta_{J,T}^2:= \nu \sum_{F\in\mcalF_T} h_F^{-1} \dnorm{\jump{\bm{r}_h^{k+1}\hhoVecvars{u}_h}_F}_F^2.
	\end{align*}
	We also identify the data oscillation for the body force function $ \bm{f} $ on $ \mcalTh $ as
	\begin{align}\label{sec_post_eq3}
		{\rm osc}^2(\bm{f},\mcalTh) := \sum_{T\in\mcalTh} {\rm osc}^2(\bm{f},T), \quad {\rm osc}^2(\bm{f},T) := \nu^{-1} h_T^2\dnorm{\bm{f}-\bm{\pi}_T^{0,k}\bm{f}}_T^2.
	\end{align}

	The main goal of the following content is to give the a posteriori error estimations with the estimator $ \eta_h $ defined in \eqref{sec_post_eq2}. Let $ (\bm{u}, p) $ and $ (\hhoVecvars{u}_h, p_h) $ be the solutions of weak form \eqref{sec_model_eq2} and discrete scheme \eqref{sec_model_eq25}, respectively, the below Helmholtz decomposition has a crucial role in our upper bound analysis.
	
	\begin{Lemma}[{Helmholtz decomposition. \citep[Lemma 3.2]{1995_MC_nonc_Stokes}}]\label{sec_post_le1}
		For $ \nabla\bm{u} - \nabla_h\bm{r}_h^{k+1}\hhoVecvars{u}_h\in [L^2(\Omega)]^{d\times d} $, there exists $ \bm{z}\in\sqd{H^1_0(\Omega)} $ satisfying $ \nabla\cdot \bm{z}=0 $, $ q\in L_0^2(\Omega) $, and $ \bm{s}\in [H^1(\Omega)\cap L_0^2(\Omega)]^d $, such that
		\begin{align}\label{sec_post_eq4}
			\nabla\bm{u} - \nabla_h\bm{r}_h^{k+1}\hhoVecvars{u}_h = \nabla \bm{z} - \nu^{-1}q\identityM + \nu^{-1}\bcurl \bm{s}
		\end{align}
		(where $ \identityM\in \mathbb{R}^{d\times d} $ is the identity matrix) and that
		\begin{align}\label{sec_post_eq5}
			\nu\dnorm{\nabla \bm{z}} + \dnorm{q} + \dnorm{\nabla \bm{s}} \leq \nu\dnorm{\nabla\bm{u} - \nabla_h\bm{r}_h^{k+1}\hhoVecvars{u}_h}.
		\end{align}
	\end{Lemma}
	
	As a result, the lemmas and theorem below, give the upper bound of the discretization errors.
	\begin{Lemma}\label{sec_post_le5}
		Let $ \bm{u} $ and $ \hhoVecvars{u}_h $ denote the unique velocity solutions to problems \eqref{sec_model_eq2} and \eqref{sec_model_eq25} and, $ \bm{e_u} $ be the error defined in \eqref{sec_model_eq26_1}. The following upper bound estimate holds true,
		\begin{align}
			\bm{e_u}^2 \lesssim \eta_h^2 + {\rm osc}^2(\bm{f},\mcalTh), 
		\end{align}
		with hidden constant independent of both the meshsize and the problem data.
	\end{Lemma}
	\begin{proof}
		By using Lemma \ref{sec_post_le1}, it follows that 
		\begin{align*}
			&\nu\dnorm{\nabla \bm{u} - \nabla_h\bm{r}_h^{k+1}\hhoVecvars{u}_h}^2 \\
			&= \nu(\nabla \bm{u} - \nabla_h\bm{r}_h^{k+1} \hhoVecvars{u}_h, \nabla\bm{z}) - (\nabla \bm{u} - \nabla_h\bm{r}_h^{k+1} \hhoVecvars{u}_h, q\identityM) + (\nabla \bm{u} - \nabla_h\bm{r}_h^{k+1} \hhoVecvars{u}_h, \bcurl \bm{s}) \\
			&=: \mathfrak{T}_1 + \mathfrak{T}_2 + \mathfrak{T}_3.
		\end{align*}
		
		We first bound $ \mathfrak{T}_1 $. Recalling the weak formulation \eqref{sec_model_eq2}, the definition of elliptic projector \eqref{sec_model_eq6}, the property of $ \hhoVecvars{I}_T^{k+1} $ \eqref{sec_model_eq19}, the property $ \nabla\cdot\bm{z}=0 $ and the discrete problem \eqref{sec_model_eq25}, we have
		\begin{align*}
			\mathfrak{T}_1 &= \nu(\nabla\bm{u},\nabla\bm{z}) - \nu(\nabla_h\bm{r}_h^{k+1}\hhoVecvars{u}_h, \nabla\bm{z}) \\
			&= (\bm{f},\bm{z}) - (\nabla\cdot\bm{z}, p) - \nu(\nabla_h\bm{r}_h^{k+1}\hhoVecvars{u}_h, \nabla_h\bm{\pi}_h^{1,k+1}\bm{z}) \\
			&= (\bm{f},\bm{z}) - \nu(\nabla_h\bm{r}_h^{k+1}\hhoVecvars{u}_h, \nabla_h\bm{r}_h^{k+1}\hhoVecvars{I}_h^k\bm{z}) \\
			&= (\bm{f},\bm{z}) - \big[ (\bm{f}, \bm{\pi}_h^{0,k}\bm{z}) - \nu\sum_{T\in\mcalTh}s_T(\hhoVecvars{u}_T, \hhoVecvars{I}_T^k\bm{z}) + \sum_{T\in\mcalTh}(D_T^k\hhoVecvars{I}_T^k\bm{z}, p_T)_T \big] \\
			&= (\bm{f}, \bm{z}-\bm{\pi}_h^{0,k}\bm{z}) + \nu\sum_{T\in\mcalTh}s_T(\hhoVecvars{u}_T,\hhoVecvars{I}_T^k\bm{z}) - \sum_{T\in\mcalTh}(\pi_T^{0,k}(\nabla\cdot\bm{z}), p_T)_T \\
			&= (\bm{f}-\bm{\pi}_h^{0,k}\bm{f}, \bm{z}-\bm{\pi}_h^{0,k}\bm{z}) + \nu\sum_{T\in\mcalTh}s_T(\hhoVecvars{u}_T,\hhoVecvars{I}_T^k\bm{z}) \\
			&\leq \sum_{T\in\mcalTh}h_T\dnorm{\bm{f}-\bm{\pi}_T^{0,k}\bm{f}}_T\dnorm{\nabla \bm{z}}_T + \nu\sum_{T\in\mcalTh}s_T(\hhoVecvars{u}_T,\hhoVecvars{u}_T)^{\frac{1}{2}}s_T(\hhoVecvars{I}_T^k\bm{z},\hhoVecvars{I}_T^k\bm{z})^{\frac{1}{2}} \\
			&\lesssim \Big(\sum_{T\in\mcalTh}\nu^{-\frac{1}{2}}\big( {\rm osc}(\bm{f},T) + \eta_{s,T} \big)\Big)(\nu\dnorm{\nabla\bm{u} - \nabla_h\bm{r}_h^{k+1}\hhoVecvars{u}_h}),
		\end{align*}
		where we have used \eqref{sec_post_eq5} and the property \eqref{sec_model_eq22_1} of $ s_T $, by adding all elements together, such that $ s_h(\hhoVecvars{I}_h^k\bm{z},\hhoVecvars{I}_h^k\bm{z})^{\frac{1}{2}}\lesssim \dnorm{\nabla \bm{z}} \leq \dnorm{\nabla\bm{u} - \nabla_h\bm{r}_h^{k+1}\hhoVecvars{u}_h} $ to pass to the penultimate line. 
		
		Using \eqref{sec_model_eq1b} and \eqref{sec_post_eq5}, we directly have 
		\begin{align*}
			\mathfrak{T}_2 &= - (\nabla\bm{u}, q\identityM) + (\nabla_h\bm{r}_h^{k+1}\hhoVecvars{u}_h, q\identityM) \\
			&= - (\nabla\cdot\bm{u}, q) + \sum_{T\in\mcalTh}(\nabla\cdot\bm{r}_T^{k+1}\hhoVecvars{u}_T, q)_T \\
			&\leq \Big(\sum_{T\in\mcalTh}\nu^{-\frac{1}{2}}\eta_{d,T}\Big) (\nu\dnorm{\nabla\bm{u} - \nabla_h\bm{r}_h^{k+1}\hhoVecvars{u}_h}).
		\end{align*}
		
		We are now in a position to get the boundedness of $ \mathfrak{T}_3 $, by using the integration by parts, the derivation follows,
		\begin{align*}
			\mathfrak{T}_3 &= \sum_{T\in\mcalT_h}(\nabla (\bm{u} - \bm{r}_T^{k+1} \hhoVecvars{u}_T), \bcurl \bm{s} )_T\\
			&= \sum_{T\in\mcalT_h} \Big( -\underset{=\, 0}{\underbrace{(\bm{u}-\bm{r}_T^{k+1}\hhoVecvars{u}_T, \nabla\cdot\bcurl \bm{s})_T} } + \sum_{F\in\mcalF_T}(\bm{u}-\bm{r}_T^{k+1}\hhoVecvars{u}_T, \bcurl\bm{s}\,\bm{n}_{TF})_F \Big) \\
			&= \underset{=\, 0}{\underbrace{\sum_{T\in\mcalT_h}\sum_{F\in\mcalF_T}(\bm{u}, \bcurl\bm{s}\,\bm{n}_{TF})_F} }   + \sum_{T\in\mcalT_h}\sum_{F\in\mcalF_T}(-\bm{r}_T^{k+1}\hhoVecvars{u}_T, \bcurl\bm{s}\,\bm{n}_{TF})_F \\
			&= \sum_{F\in\mcalF_h}(\jump{\bm{r}_h^{k+1}\hhoVecvars{u}_h}_F, \bcurl\bm{s}\,\bm{n}_{F})_F \\
			&\leq \sum_{F\in\mcalF_h} \Big(\dnorm{\jump{\bm{r}_h^{k+1}\hhoVecvars{u}_h}_F}_{\frac{1}{2},2,F} \dnorm{\bcurl\bm{s}\,\bm{n}_{F}}_{-\frac{1}{2},2,F} \Big) \\
			&\lesssim \Big( \sum_{F\in\mcalF_h} h_F^{-\frac{1}{2}}\dnorm{\jump{\bm{r}_h^{k+1}\hhoVecvars{u}_h}_F}_F \Big) \Big( \sum_{F\in\mcalF_h} \sum_{T\in\mcalT_F}\dnorm{\bcurl\bm{s}}_{{\rm div},T} \Big) \\
			&\lesssim \Big( \sum_{T\in\mcalTh}\nu^{-\frac{1}{2}} (\nu \sum_{F\in\mcalF_T} h_F^{-1}\dnorm{\jump{\bm{r}_h^{k+1}\hhoVecvars{u}_h}_F}_F^2 )^{\frac{1}{2}} \Big) \Big( \sum_{T\in\mcalTh}\dnorm{\bcurl\bm{s}}_T \Big) \\
			&\lesssim \Big(\sum_{T\in\mcalTh} \nu^{-\frac{1}{2}} \eta_{J,T} \Big)(\nu\dnorm{\nabla\bm{u} - \nabla_h\bm{r}_h^{k+1}\hhoVecvars{u}_h}),
		\end{align*}
		where we have used the inverse Sobolev embedding inequality (c.f., \citep[Corollary 1.29]{2020_hho_book}) and the fact $ \bcurl\bm{s}\in \bm{H}({\rm div},\Omega) $ (defined by \eqref{sec_model_eq1_add1}) to pass to the sixth line.
		
		Collecting the above estimates and the stability \eqref{sec_model_eq22}, the conclusion is now straightforward.
		\qedhere
	\end{proof}
	
	\begin{Lemma}
		Let $ p $ and $ p_h $ denote the unique pressure solutions to problems \eqref{sec_model_eq2} and \eqref{sec_model_eq25} and, $ e_p $ be the error defined in \eqref{sec_model_eq26_1}. The following upper bound estimate holds true,
		\begin{align}
			e_p^2 \lesssim \eta_h^2 + {\rm osc}^2(\bm{f},\mcalTh), 
		\end{align}
		with hidden constant independent of both the meshsize and the problem data.
	\end{Lemma}
	\begin{proof}
		We deduce from the scheme \eqref{sec_model_eq25a}, the property \eqref{sec_model_eq19} and \eqref{sec_model_eq22_1} that for all $ \bm{v}\in\sqd{H^1_0(\Omega)} $, we write 
		\begin{align*}
			(\nabla\cdot\bm{v}, p-p_h) &= (\nabla\bm{u},\nabla\bm{v}) - (\bm{f},\bm{v}) - (\nabla\cdot\bm{v}, p_h) \\
			&= (\nabla\bm{u},\nabla\bm{v}) - (\bm{f},\bm{v}) -(\bm{\pi}_h^{0,k}(\nabla\cdot\bm{v}),p_h)\\
			&= (\nabla\bm{u},\nabla\bm{v}) - (\bm{f},\bm{v}) - (D_h^k\hhoVecvars{I}_h^k\bm{v},p_h) \\
			&= (\nabla\bm{u},\nabla\bm{v}) - (\bm{f},\bm{v}) + (\bm{f},\bm{\pi}_h^{0,k}\bm{v}) - \nu(\nabla_h\bm{r}_h^{k+1}\hhoVecvars{u}_h,\nabla_h\bm{r}_h^{k+1}\hhoVecvars{u}_h) - \nu s_h(\hhoVecvars{u}_h,\hhoVecvars{I}_h^k\bm{v})\\
			&= (\bm{f}-\bm{\pi}_h^{0,k}\bm{f}, \bm{\pi}_h^{0,k}\bm{v}-\bm{v}) + \nu(\nabla\bm{u}-\nabla_h\bm{r}_h^{k+1}\hhoVecvars{u}_h, \nabla\bm{v}) - \nu s_h(\hhoVecvars{u}_h,\hhoVecvars{I}_h^k\bm{v})\\
			&\lesssim \Big(\sum_{T\in\mcalTh}h_T\dnorm{\bm{f}-\bm{\pi}_T^{0,k}\bm{f}}_T \Big) \dnorm{\nabla\bm{v}} + \nu\dnorm{\nabla\bm{u}-\nabla_h\bm{r}_h^{k+1}\hhoVecvars{u}_h}\dnorm{\nabla\bm{v}} + \nu s_h(\hhoVecvars{u}_h,\hhoVecvars{u}_h)^{\frac{1}{2}}\dnorm{\nabla \bm{v}}.
		\end{align*}
		Noting that $ \dnorm{p-p_h} \in Q $ and by using the inf-sup condition \eqref{sec_model_eq4}, it holds
		\begin{align*}
			\nu^{-1}\dnorm{p-p_h}^2 \lesssim {\rm osc}^2(\bm{f},\mcalTh) + \eta_{s,h}^2 + \nu\dnorm{\nabla\bm{u}-\nabla_h\bm{r}_h^{k+1}\hhoVecvars{u}_h}^2,
		\end{align*}
		which, involving Lemma \ref{sec_post_le5}, gives the conclusion.
		\qedhere
	\end{proof}
	
	As an immediate consequence of the two lemmas above, we obtain the following theorem.
	\begin{Theorem}[Upper bound]\label{sec_post_th1}
		Let $ (\bm{u}, p) $ and $ (\hhoVecvars{u}_h, p_h) $ denote the unique solutions to problems \eqref{sec_model_eq2} and \eqref{sec_model_eq25}. Let $ \bm{e_u} $ and $ e_p $ be the errors defined in \eqref{sec_model_eq26_1}. The following upper bound estimate holds:
		\begin{align}
			\bm{e_u}^2 + e_p^2 \lesssim \eta_h^2 + {\rm osc}^2(\bm{f},\mcalTh),
		\end{align}
		where the hidden constant may be dependent on the stability constants $ \beta $ (cf., \eqref{sec_model_eq4}), $ \lambda $ (cf., \eqref{sec_model_eq22}) and $ \widetilde{\beta} $ (cf., \eqref{sec_model_eq24}), but independent of $ \bm{u} $, $ p $, $ h $ and $ \nu $.
	\end{Theorem}
	
	In addition to the above-mentioned upper bound, we have the lower bound as follows.
	
	\begin{Theorem}[Local lower bound] Under the assumption of Theorem \ref{sec_post_th1} and for the fixed $ T\in\mcalTh $, we have the local lower bound 
		\begin{align*}
			\eta_{d,T}^2 \leq \bm{e}_{\bm{u},T}^2, \quad \eta_{s,T}^2 \leq \bm{e}_{\bm{u},T}^2, \quad \eta_{J,T}^2 \lesssim \sum_{T'\in\mcalT_{N,T}} \bm{e}_{\bm{u},T'}^2.
		\end{align*}
	\end{Theorem}
	\begin{proof}
		Clearly, from the definitions \eqref{sec_model_eq26_1} and \eqref{sec_post_eq2} of $ \bm{e}_{\bm{u},T} $ and $ \eta_{s,T} $, respectively, one has $ \eta_{s,T}^2 \leq \bm{e}_{\bm{u},T}^2 $. With the fact that $ \nabla\cdot\bm{u} = 0 $, we get 
		\begin{align*}
			\eta_{d,T}^2 = \nu \dnorm{\nabla\cdot\bm{r}_T^{k+1}\hhoVecvars{u}_T}_T^2  = \nu \dnorm{\nabla\cdot(\bm{u} - \bm{r}_T^{k+1}\hhoVecvars{u}_T)}_T^2 \leq \nu \dnorm{\nabla(\bm{u} - \bm{r}_T^{k+1}\hhoVecvars{u}_T)}_T^2 = \bm{e}_{\bm{u},T}^2.
		\end{align*}
		
		It remains to bound $ \eta_{J,T}^2 =  \nu\sum_{F\in\mcalF_T}h_F^{-1}\dnorm{\jump{\bm{r}_h^{k+1}\hhoVecvars{u}_h}_F}_F^2 $. Noting that, for all $ F\in\mcalFh^i $ with bordering elements $ T_1 $ and $ T_2 $, $ \jump{\bm{r}_h^{k+1}\hhoVecvars{u}_h}_F = (\bm{r}_{T_1}^{k+1}\hhoVecvars{u}_{T_1} - \bm{u}_F) - (\bm{r}_{T_2}^{k+1}\hhoVecvars{u}_{T_2} - \bm{u}_F) $ and, for all $ F\in\mcalFh^b $ with bordering element $ T $, $ \jump{\bm{r}_h^{k+1}\hhoVecvars{u}_h}_F = \bm{r}_{T}^{k+1}\hhoVecvars{u}_{T} - \bm{u}_F $ owing to $ \bm{u}_F=\bm{0} $. Similarly, for all $ \bm{u}\in\sqd{H_0^1(\Omega)} $ and all $ F\in\mcalFh $, $ \jump{\bm{r}_h^{k+1}\hhoVecvars{u}_h}_F = \jump{\bm{r}_h^{k+1}\hhoVecvars{u}_h - \bm{u}}_F $. Inserting $ \bm{\pi}_F^{0,k}\jump{\bm{r}_h^{k+1}\hhoVecvars{u}_h}_F - \bm{\pi}_F^{0,k}\jump{\bm{r}_h^{k+1}\hhoVecvars{u}_h - \bm{u}}_F = 0 $ and using the triangle inequality, we obtain
		\begin{align*}
			&\nu\sum_{F\in\mcalF_T}h_F^{-1}\dnorm{\jump{\bm{r}_h^{k+1}\hhoVecvars{u}_h}_F}_F^2 \\
			& \leq 2\nu \sum_{F\in\mcalF_T}h_F^{-1} \big( \dnorm{ \jump{\bm{r}_h^{k+1}\hhoVecvars{u}_h - \bm{u}}_F  - \bm{\pi}_F^{0,k} \jump{ \bm{r}_h^{k+1}\hhoVecvars{u}_h-\bm{u} }_F }_F^2 + \dnorm{\bm{\pi}_F^{0,k}\jump{\bm{r}_h^{k+1}\hhoVecvars{u}_h}_F}^2_F  \big) \\
			& \leq 2 \nu \sum_{F\in\mcalF_T}\sum_{T'\in\mcalT_F} h_F^{-1}\Big(\dnorm{(\bm{r}_{T'}^{k+1}\hhoVecvars{u}_{T'}-\bm{u} )-\bm{\pi}_F^{0,k}(\bm{r}_{T'}^{k+1}\hhoVecvars{u}_{T'}-\bm{u} )}_F^2 + \dnorm{\bm{\pi}_F^{0,k}(\bm{r}_{T'}^{k+1}\hhoVecvars{u}_{T'} - \bm{u}_F) }_F^2\Big) \\
			&\leq 2 \nu \sum_{T'\in\mcalT_{N,T}} \sum_{F\in\mcalF_{T'}}h_F^{-1}\Big(\dnorm{(\bm{r}_{T'}^{k+1}\hhoVecvars{u}_{T'}-\bm{u} )-\bm{\pi}_F^{0,k}(\bm{r}_{T'}^{k+1}\hhoVecvars{u}_{T'}-\bm{u} )}_F^2 + \dnorm{\bm{\pi}_F^{0,k}(\bm{r}_{T'}^{k+1}\hhoVecvars{u}_{T'} - \bm{u}_F) }_F^2\Big)  \\
			&\lesssim \sum_{T'\in\mcalT_{N,T}} \bm{e}_{\bm{u},T'}^2,
		\end{align*}
		where we have used the orthogonal approximation property \eqref{sec_model_eq16_1} and the definition of $ s_T $ to pass to the penultimate line.
		
		The conclusion is now straightforward.
		\qedhere
	\end{proof} 
	Summing over $ T\in\mcalTh $, we have the following lower bound for the error estimator.
	\begin{Theorem}[Lower bound] Under the assumption of Theorem \ref{sec_post_th1}, we have the lower bound 
		\begin{align}
			\eta_{h}^2 \lesssim \bm{e_u}^2 + e_p^2 + {\rm osc}^2(\bm{f},\mcalTh).
		\end{align}
	\end{Theorem}
	\comment{
		\begin{proof}
			Clearly, from the definitions \eqref{sec_model_eq26_1} and \eqref{sec_post_eq2} of $ \bm{e_u} $ and $ \eta_{s,h} $, respectively, one has $ \eta_{s,h}^2 \leq \bm{e_u}^2 $. With the fact that $ \nabla\cdot\bm{u} = 0 $, we get 
			\begin{align*}
				\eta_{d,h}^2 = \nu \sum_{T\in\mcalT_h}\dnorm{\nabla\cdot\bm{r}_T^{k+1}\hhoVecvars{u}_T}_T^2  = \nu \sum_{T\in\mcalT_h}\dnorm{\nabla\cdot(\bm{u} - \bm{r}_T^{k+1}\hhoVecvars{u}_T)}_T^2 \leq \nu \sum_{T\in\mcalT_h}\dnorm{\nabla(\bm{u} - \bm{r}_T^{k+1}\hhoVecvars{u}_T)}_T^2.
			\end{align*}
			
			It remains to bound $ \eta_{J,h}^2 =  \nu\sum_{F\in\mcalF_T}h_F^{-1}\dnorm{\jump{\bm{r}_h^{k+1}\hhoVecvars{u}_h}_F}_F^2 $. Noting that, for all $ F\in\mcalFh^i $ with bordering elements $ T_1 $ and $ T_2 $, $ \jump{\bm{r}_h^{k+1}\hhoVecvars{u}_h}_F = (\bm{r}_{T_1}^{k+1}\hhoVecvars{u}_{T_1} - \bm{u}_F) - (\bm{r}_{T_2}^{k+1}\hhoVecvars{u}_{T_2} - \bm{u}_F) $ and, for all $ F\in\mcalFh^b $ with bordering element $ T $, $ \jump{\bm{r}_h^{k+1}\hhoVecvars{u}_h}_F = \bm{r}_{T}^{k+1}\hhoVecvars{u}_{T} - \bm{u}_F $ owing to $ \bm{u}_F=\bm{0} $. Similarly, for all $ \bm{u}\in\sqd{H_0^1(\Omega)} $ and all $ F\in\mcalFh $, $ \jump{\bm{r}_h^{k+1}\hhoVecvars{u}_h}_F = \jump{\bm{r}_h^{k+1}\hhoVecvars{u}_h - \bm{u}}_F $. Inserting $ \bm{\pi}_F^{0,k}\jump{\bm{r}_h^{k+1}\hhoVecvars{u}_h}_F - \bm{\pi}_F^{0,k}\jump{\bm{r}_h^{k+1}\hhoVecvars{u}_h - \bm{u}}_F = 0 $ and using the triangle inequality, we obtain
			\begin{align*}
				\nu\sum_{F\in\mcalF_h}h_F^{-1}\dnorm{\jump{\bm{r}_h^{k+1}\hhoVecvars{u}_h}_F}_F^2 &\leq  2\nu\sum_{T\in\mcalTh} \sum_{F\in\mcalF_T}h_F^{-1}\dnorm{\jump{\bm{r}_h^{k+1}\hhoVecvars{u}_h}_F}_F^2 \\
				& \leq 4\nu\sum_{T\in\mcalTh} \sum_{F\in\mcalF_T}h_F^{-1} \big( \dnorm{ \jump{\bm{r}_h^{k+1}\hhoVecvars{u}_h - \bm{u}}_F  - \bm{\pi}_F^{0,k} \jump{ \bm{r}_h^{k+1}\hhoVecvars{u}_h-\bm{u} }_F }_F^2 + \dnorm{\bm{\pi}_F^{0,k}\jump{\bm{r}_h^{k+1}\hhoVecvars{u}_h}_F}^2_F  \big) \\
				& \leq 8 \nu\sum_{T\in\mcalTh} \sum_{F\in\mcalF_T} h_F^{-1}\dnorm{(\bm{r}_{T}^{k+1}\hhoVecvars{u}_{T}-\bm{u} )-\bm{\pi}_F^{0,k}(\bm{r}_{T}^{k+1}\hhoVecvars{u}_{T}-\bm{u} )}_F^2  \\ 
				&\hspace{1.1em} + 8 \nu\sum_{T\in\mcalTh} \sum_{F\in\mcalF_T} h_F^{-1}\dnorm{\bm{\pi}_F^{0,k}(\bm{r}_{T}^{k+1}\hhoVecvars{u}_{T} - \bm{u}_F) }_F^2 \\
				&\lesssim \nu\sum_{T\in\mcalT_h}\dnorm{\nabla(\bm{u} - \bm{r}_T^{k+1}\hhoVecvars{u}_T)}_T^2 + \nu s_h(\hhoVecvars{u}_h,\hhoVecvars{u}_h),
			\end{align*}
			where we have used the orthogonal approximation property \eqref{sec_model_eq16_1} and the definition of $ s_h $ to pass to the last second line.
			
			By adding the above estimations together, the conclusion is now straightforward.
			\qedhere
		\end{proof}
	}

	\section{Nonhomogeneous Dirichlet boundary condition}
	In this section, we give a short comment for the nonhomogeneous Dirichlet boundary condition. Except for the boundary conditions, we keep the same settings as before, now, we consider the following problem
	\begin{subequations}\label{sec_nonDir_eq1}
		\begin{align}
			\label{sec_nonDir_eq1a}
			-\nu \Delta \bm{u} + \nabla p &= \bm{f} \quad \text{in}\ \Omega, \\ 
			\label{sec_nonDir_eq1b}
			\nabla\cdot\bm{u} &= 0 \quad \,\text{in}\ \Omega, \\ 
			\label{sec_nonDir_eq1c}
			\bm{u} &= \bm{g} \ \,\text{on}\ \partial \Omega, 
		\end{align}
	\end{subequations}
	where $ \bm{g} $ satisfies the compatibility condition $ \int_{\partial\Omega} \bm{g} \cdot \bm{n} = 0 $ ($ \bm{n} $ is the unit vector normal to $ \partial \Omega $). 
	
	Consider $ \bm{u}_D\in \sqd{H^1(\Omega)} $, satisfying 
	\begin{subequations}\label{sec_nonDir_eq2}
		\begin{align}
			\nabla\cdot\bm{u}_D &= 0 \quad \text{in}\ \Omega, \\
			\bm{u}_D &= \bm{g} \quad \text{on}\ \partial \Omega.
		\end{align}
	\end{subequations}
	According to the Helmholtz decomposition of a vector field in $ \sqd{L^2(\Omega)} $, the solution of \eqref{sec_nonDir_eq2} exists. Then, a weak solution $ (\bm{u},p)\in [H^1(\Omega)]^d\times Q $ to problem \eqref{sec_nonDir_eq1} can be obtained as $ \bm{u} = \bm{u}_0 + \bm{u}_D $, where $ \bm{u}_0\in \bm{V} $ is such that
	\begin{align*}
		\nu (\nabla\bm{u}_0, \nabla\bm{v}) - (\nabla\cdot\bm{v}, p) &= (\bm{f},\bm{v}) - \nu(\nabla \bm{u}_D,\nabla\bm{v}) \quad \forall \bm{v}\in \bm{V}, \\
		(\nabla\cdot\bm{u}_0, q) &= 0 \quad \forall q \in Q.
	\end{align*}
	
	Next, we consider the HHO solution for discrete problem \eqref{sec_model_eq25}. Let $ \hhoVecvars{u}_{h,D} := ((\bm{u}_{T,D})_{T\in\mcalT_h}, (\bm{u}_{F,D})_{F\in\mcalF_h})\in\hhoVecvars{V}_h^k $ be such that
	\begin{align*}
		\bm{u}_{T,D} = \bm{0} \ \ \forall T\in\mcalT_h, \quad \bm{u}_{F,D} = \bm{\pi}_F^{0,k} \bm{g} \ \ \forall F\in\mcalF_h^b, \quad \bm{u}_{F,D} = \bm{0} \ \ \forall F\in \mcalF_h^i.
	\end{align*}
	Then, the HHO solution $ (\hhoVecvars{u}_h, p_h)\in \hhoVecvars{V}_h^k\times Q_h $ is obtained as $ \hhoVecvars{u}_h = \hhoVecvars{u}_{h,0} + \hhoVecvars{u}_{h,D} $ with $ \hhoVecvars{u}_{h,0}\in\hhoVecvars{V}_{h,0}^k $ such that 
	\begin{subequations}\label{sec_nonDir_eq3}
		\begin{align}
			\label{sec_nonDir_eq3a}
			\nu a_h(\hhoVecvars{u}_{h,0}, \hhoVecvars{v}_h) + b_h(\hhoVecvars{v}_h, p_h) &= (\bm{f}, \bm{v}_h) - \nu a_h(\hhoVecvars{u}_{h,D}, \hhoVecvars{v}_h) \quad \forall \hhoVecvars{v}_h\in\hhoVecvars{V}_{h,0}^k, \\
			\label{sec_nonDir_eq3b}
			-b_h(\hhoVecvars{u}_{h,0}, q_h) &= 0 \hspace{3.5em} \forall q_h \in Q_h^k. 
		\end{align}
	\end{subequations}

	\section{Numerical examples}\label{num_exams}
	In this section, we shall present numerical results to demonstrate the accuracy of the theoretical estimates in the previous sections and the effectiveness of the proposed estimator. The eﬀectivity index is defined as  
	\begin{align}\label{sec_NE_eq1}
		\text{eff}:= \left( \frac{\bm{e_u}^2 + e_p^2}{\eta_h^2} \right)^{\frac{1}{2}}.
	\end{align}
	The constant behavior of the eﬀectivity index will be checked in the following numerical examples, which shows that the constants of equivalence between exact and estimated errors are independent of the meshsize.
	
	The mesh elements marked strategy follows from the D\"{o}rfler criterion 
	\begin{align}\label{sec_NE_eq2}
		\sum_{T\in\mathcal{M}_h} \eta_T^2 \geq \theta \sum_{T\in\mcalTh}\eta_T^2, \quad 0 < \theta < 1,
	\end{align}
	in which, the elements with larger error indicators are put into the marked set $ \mathcal{M}_h $ until the condition \eqref{sec_NE_eq2} is satisfied. Then the marked elements are refined to obtain a new mesh by connecting the barycenter of each element to the mid-point of each edge corresponding to this element. This method was also used in \cite{2017_Cangiani_posteri,2019_SIAM_BaoWG}. We should also note that, in the above mentioned refinement, all the marked elements will be refined into several quadrilaterals. 
	
	The following numerical examples are implemented through FEALPy \cite{fealpy_publish}, and the marking parameter $ \theta $ is set to be $ 0.3 $. The convergence rates of errors with respect to degrees of freedom denoted by $ \sharp $Dof, and it is easy to find that $ h=\mathcal{O}(\sharp\text{Dof}^{-\frac{1}{2}}) $ for two-dimensional quasiuniform mesh.

	\subsection{Example 1}\label{num_example1}
	Let domain $ \Omega=(0,1)\times(0,1) $ and the exact solution be set as in \cite{2005_JSC_mixedDGpost}:
	\begin{align*}
		\bm{u} &= \begin{pmatrix}
			-e^x(y\cos(y)+\sin(y)) \\
			e^xy\sin(y)
		\end{pmatrix}, \\
		p &= 2e^x\sin(y) - ( 2(1-e)(\cos(1)-1) ).
	\end{align*}
	In this example, by fixing the polynomial order $ k=1 $ and $ \nu = 1 $, we test the different initial meshes as exhibited in Figure \ref{exam1_mesh}, and the generation of these meshes has been described in detail in \cite{2020_JCP_ZhangSD}. Since the analytical solution is smooth, we expect an optimal convergence in terms $ \eta_h $, $ \bm{e_u} $ and $ e_p $. The eﬀectivity index, errors and the corresponding convergence orders on different meshes are listed in Table \ref{numExam1_tb_allmesh1}. From the table, we can observe that the orders of $ \eta_h $, $ \bm{e_u} $ and $ e_p $ converge to the half of $ 2 $ (i.e. $ k+1 $) with respect to the degrees of freedom and, the efﬁciency index is close to constant $ 1 $. These results agree with our analytical predictions.
	
	\begin{figure}[!ht]
		\begin{center}
			\subfigure[$ \mathcal{T}_{h}^{1} $]{
				\label{exam1_mesh1}
				\centering
				\includegraphics[width=1.5in]{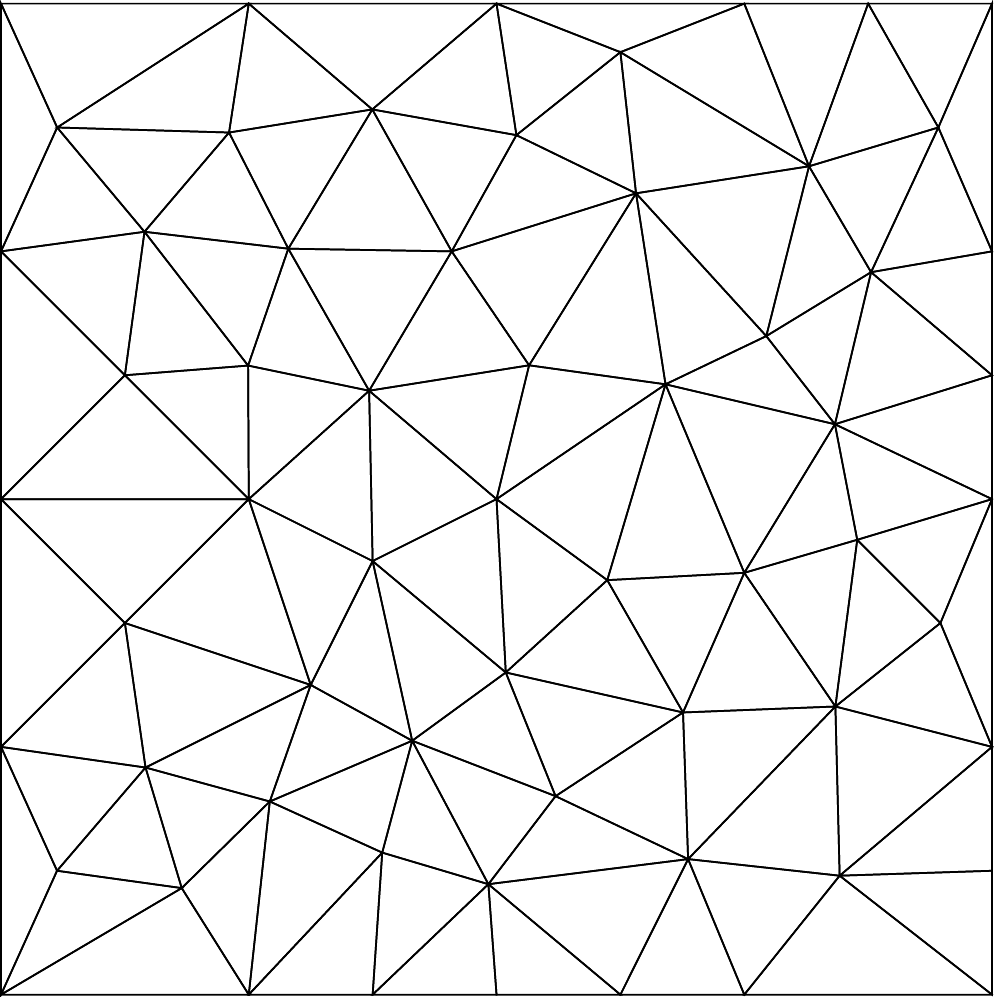}      
			}
			\subfigure[$ \mathcal{T}_{h}^{2} $ ]{
				\label{exam1_mesh2}
				\centering
				\includegraphics[width=1.5in]{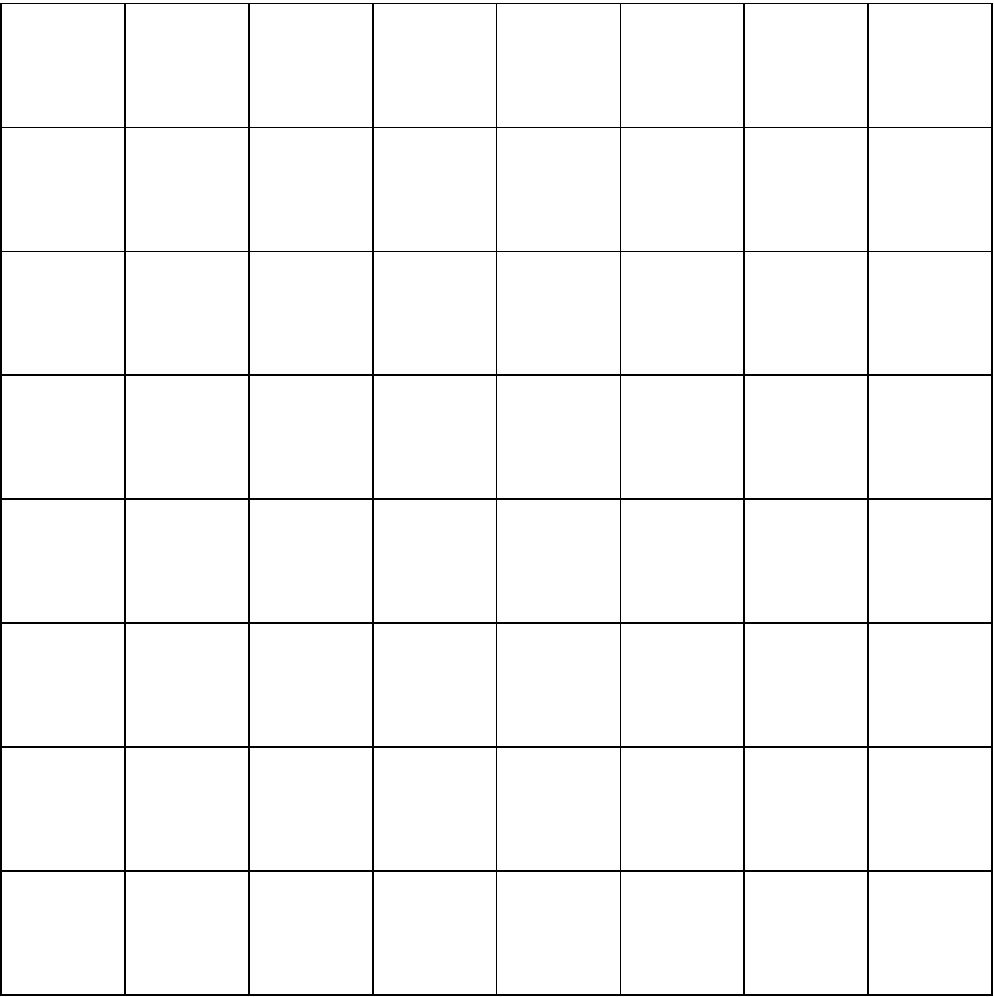}     
			}
			\subfigure[$ \mathcal{T}_{h}^{3} $]{
				\label{exam1_mesh3}
				\centering
				\includegraphics[width=1.5in]{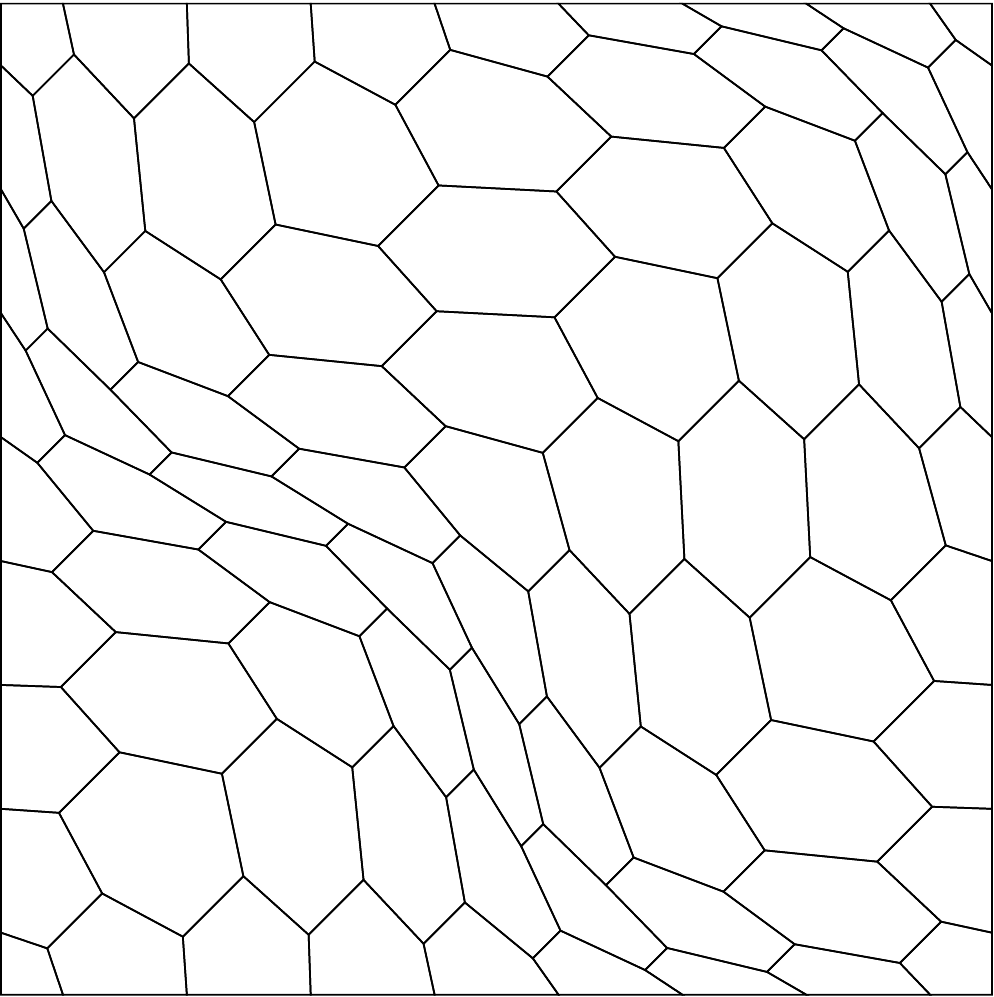}     
			}
			\subfigure[$ \mathcal{T}_{h}^{4} $]{
				\label{exam1_mesh4}
				\centering
				\includegraphics[width=1.5in]{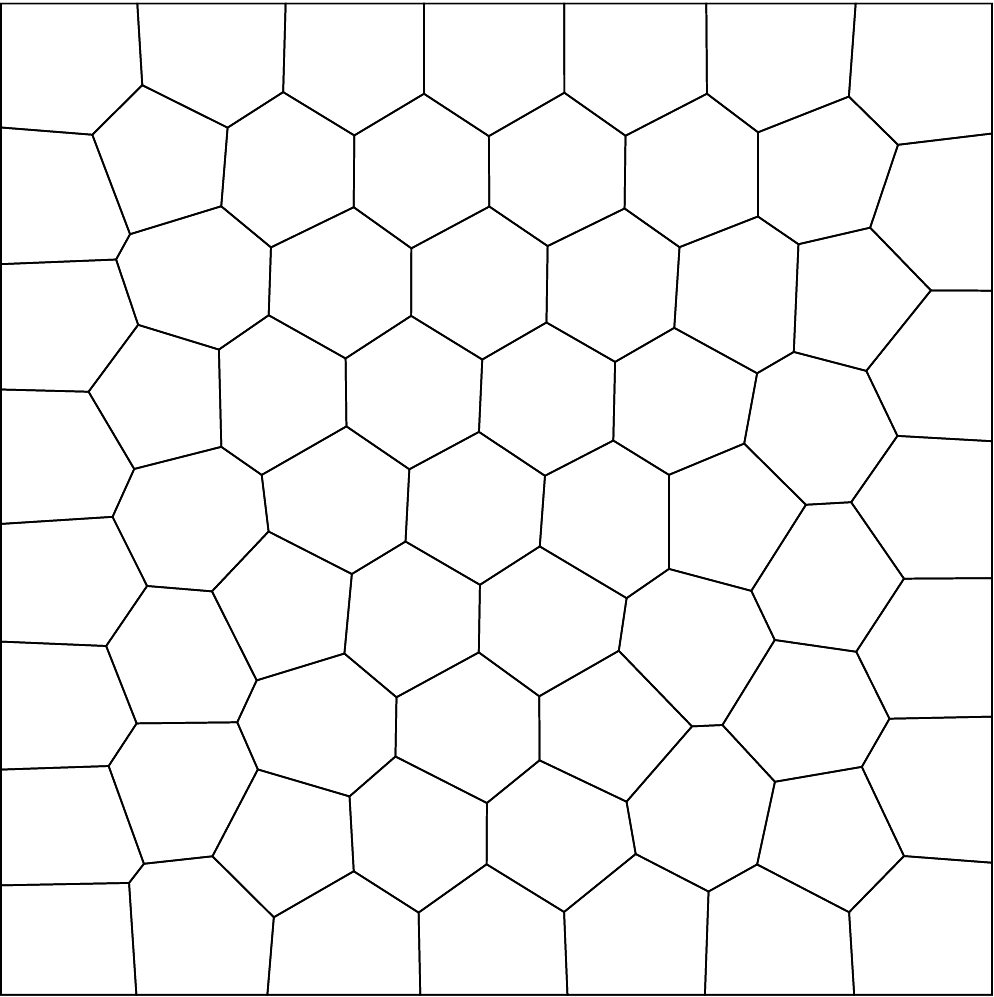}      
			}
		\end{center}
		\vspace{-1em}
		\caption{Example 1: Illustrations of meshes.} \label{exam1_mesh}
	\end{figure}

	\begin{table}[h!]
		\caption{Example 1: The errors for a series of the meshes $ \mathcal{T}_{h}^{1} - \mathcal{T}_{h}^{4} $.}
		\label{numExam1_tb_allmesh1}
		\medskip
		\centering
		\begin{tabular}[b]{ @{\  \ }c@{\ \ }| @{\ \ }  c  @{\ \ }  c  @{\ \ }  c  @{\ \ }  c  @{\ \ }  c  @{\ \ } | @{\ \  }  c  @{\ \  }  c  @{\ \ }  c  @{\  \ } | c  @{\  \ } }
			\hline
			$ \mathcal{T}_{h} $& $ \sharp\text{Dof} $& $ \eta_h $&  \text{order}&  $ \bm{e_u} $& \text{order}& $ \sharp\text{Dof} $&  $ e_p $& \text{order}&  \text{eff}  \\ 
			\hline
			\multirow{5}{*}{$ \mathcal{T}_{h}^{1} $}&  208&  2.0576e-02& --&  2.1886e-02&  --&  48&  4.2086e-03&  --&  1.0832 \\
			&  704&  6.3603e-03&  0.98&   6.7600e-03&  0.98&  144&  1.1752e-03&  1.06&   1.0788  \\
			&  2752&  1.7418e-03&  0.95&  1.8462e-03&  0.96&  576&  3.4250e-04&  0.99&   1.0780 \\
			&  10880&  4.6650e-04&  0.96&  4.9040e-04&  0.97&  2304&  7.1956e-05&  1.14&  1.0625 \\
			&  43264&  1.1819e-04&  0.99&  1.2493e-04&  0.99&  9216&  1.2948e-05&  1.24&  1.0627 \\
			\hline
			\multirow{5}{*}{$ \mathcal{T}_{h}^{2} $}&  256&  1.6132e-02& --&  1.6942e-02&  --&  48&  7.7473e-04&  --&  1.0513 \\
			&  960&  4.1441e-03&  1.02&   4.3824e-03&  1.01&  192&  2.2254e-04&  0.94&   1.0589  \\
			&  3712&  1.0488e-03&  1.01&  1.1121e-03&  1.01&  768&  4.6632e-05&  1.15&   1.0613  \\
			&  14592&  2.6381e-04&  1.00&  2.7994e-04&  1.00&  3072&  8.6231e-06&  1.23&  1.0616 \\
			&  57856&  6.6154e-05&  1.00&  7.0211e-05&  1.00&  12288&  1.5407e-06&  1.25&  1.0616 \\
			\hline
			\multirow{5}{*}{$ \mathcal{T}_{h}^{3} $}&  502&  1.3367e-02& --&  1.4352e-02&  --&  75&  1.6235e-03&  --&  1.0805 \\
			&  1984&  3.3322e-03&  0.98&   3.4187e-03&  1.01&  384&  7.0583e-04&  1.09&   1.0475  \\
			&  7552&  8.7864e-04&  0.99&  9.2568e-04&  0.97&  1536&  1.7979e-04&  1.01&   1.0732  \\
			&  29440&  2.3146e-04&  0.98&  2.4488e-04&  0.98&  6144&  3.6668e-05&  1.17&  1.0698 \\
			&  116224&  5.9664e-05&  0.99&  6.3129e-05&  0.99&  24576&  6.3977e-06&  1.26&  1.0635 \\
			\hline
			\multirow{5}{*}{$ \mathcal{T}_{h}^{4} $}&  502&  9.9420e-03& --&  1.0633e-02&  --&  75&  2.1951e-03&  --&  1.0901 \\
			&  2024&  2.5498e-03&  0.95&   2.5967e-03&  0.99&  396&  5.4604e-04&  1.00&   1.0407  \\
			&  7744&  6.7444e-04&  0.98&  7.0717e-04&  0.97&  1584&  1.4445e-04&  0.99&   1.0702  \\
			&  30272&  1.7793e-04&  0.98&  1.8735e-04&  0.97&  6336&  3.0179e-05&  1.14&  1.0665 \\
			&  119680&  4.5649e-05&  0.99&  4.8388e-05&  0.98&  25344&  5.3445e-06&  1.25&  1.0664 \\
			\hline 
		\end{tabular}
	\end{table}

	\subsection{Example 2}\label{num_example2}
	In this example, we investigate on the efﬁciency index and the errors $ \eta_h $, $ \bm{e_u} $, $ e_p $ on mesh $ \mcalT_h^2 $ by taking polynomial orders $ k = 0, 1, 2, 3 $ and $ \nu = 1, 10^{-1}, 10^{-3}, 10^{-6}, 10^{-10} $. Considering the following exact solution in $ \Omega = (0,1)\times(0,1) $:
	\begin{align*}
		\bm{u} &=  \begin{pmatrix}
			-\frac{1}{2}\cos^2(x)\cos(y)\sin(y)\\
			\frac{1}{2}\cos^2(y)\cos(x)\sin(x)
		\end{pmatrix}, \\
		p &= x^6-y^6.
	\end{align*}
	Firstly, we fix $ \nu = 1 $ to test the different polynomial orders $ k=0,1,2,3 $. The corresponding eﬀectivity index and errors are reported in Table \ref{numExam2_tb_1}. Then, we fix $ k=3 $ to test the viscosity constant $ \nu = 10^{-1}, 10^{-3}, 10^{-6}, 10^{-10} $ with Table \ref{numExam2_tb_2} showing the eﬀectivity index and errors. In particular, in Table \ref{numExam2_tb_2}, we can clearly observe that the eﬀectivity index is independent of $ \nu $, which is also consistent with the theoretical results.
	
	\begin{table}[h!]
		\caption{Example 2: The errors for different polynomial orders.}
		\label{numExam2_tb_1}
		\medskip
		\centering
		\begin{tabular}[b]{ @{\  \ }c@{\ \ }| @{\ \ }  c  @{\ \ }  c  @{\ \ }  c  @{\ \ }  c  @{\ \ }  c  @{\ \ } | @{\ \  }  c  @{\ \  }  c  @{\ \ }  c  @{\  \ } | c  @{\  \ } }
			\hline
			$ k $& $ \sharp\text{Dof} $& $ \eta_h $&  \text{order}&  $ \bm{e_u} $& \text{order}& $ \sharp\text{Dof} $&  $ e_p $& \text{order}&  \text{eff}  \\ 
			\hline
			\multirow{5}{*}{$ 0 $}&  112&  3.9460e-01& --&  2.1886e-02&  --&  16&  4.9979e-02&  --&  0.7516 \\
			&  416&  2.7629e-01&  0.50&   2.2661e-01&  0.42&  64&  3.1702e-02&  0.35&   0.8282  \\
			&  1600&  1.3385e-01&  0.54&  1.1926e-01&  0.48&  256&  1.5324e-02&  0.54&   0.8983 \\
			&  6272&  6.4573e-02&  0.53&  6.0779e-02&  0.49&  1024&  6.4664e-03&  0.63&  0.9466 \\
			&  24832&  3.1510e-02&  0.52&  3.0605e-02&  0.50&  4096&  2.3716e-03&  0.73&  0.9742 \\
			\hline
			\multirow{5}{*}{$ 1 $}&  256&  1.0040e-01& --&  9.9698e-02&  --&  48&  6.5437e-03&  --&  0.9952 \\
			&  960&  2.6633e-02&  1.00&   2.6573e-02&  1.00&  192&  8.1796e-04&  1.57&   0.9982  \\
			&  3712&  6.7878e-03&  1.01&  6.7828e-03&  1.01&  768&  1.0243e-04&  1.54&   0.9994  \\
			&  14592&  1.7080e-03&  1.01&  1.7085e-03&  1.00&  3072&  1.3629e-05&  1.47&  1.0003 \\
			&  57856&  4.2805e-04&  1.00&  4.2841e-04&  1.00&  12288&  1.9831e-06&  1.40&  1.0008 \\
			\hline
			\multirow{5}{*}{$ 2 $}&  432&  1.1121e-02& --&  1.4352e-02&  --&  96&  4.2838e-04&  --&  1.0004 \\
			&  1632&  1.4534e-03&  1.53&   1.4547e-03&  1.53&  384&  3.8484e-05&  1.81&   1.0012  \\
			&  6336&  1.8428e-04&  1.52&  1.8444e-04&  1.52&  1536&  3.4632e-06&  1.78&   1.0010  \\
			&  24960&  2.3151e-05&  1.51&  2.3170e-05&  1.52&  6144&  3.0966e-07&  1.76&  1.0001 \\
			&  99072&  2.8866e-06&  1.51&  2.8890e-06&  1.51&  24576&  2.7553e-08&  1.75&  1.0001 \\
			\hline
			\multirow{5}{*}{$ 3 $}&  640&  7.1488e-04& --&  7.1483e-04&  --&  160&  1.6366e-05&  --&  1.0002 \\
			&  2432&  4.5901e-05&  2.06&   4.5933e-05&  2.06&  640&  6.9111e-07&  2.37&   1.0001  \\
			&  9472&  2.8953e-06&  2.03&  2.8974e-06&  2.03&  2560&  2.9867e-08&  2.31&   1.0001  \\
			&  37376&  1.8756e-07&  1.99&  1.8669e-07&  2.00&  10240&  1.3005e-09&  2.28&  1.0001 \\
			&  148480&  1.1801e-08&  2.00&  1.1798e-08&  2.00&  40960&  5.7773e-11&  2.26&  0.9999 \\
			\hline 
		\end{tabular}
	\end{table}

	\begin{table}[h!]
		\caption{Example 2: The errors for $ \nu = 10^{-1}, 10^{-3}, 10^{-6}, 10^{-10} $.}
		\label{numExam2_tb_2}
		\medskip
		\centering
		\begin{tabular}[b]{ @{\  \ }c@{\ \ }| @{\ \ }  c  @{\ \ }  c  @{\ \ }  c  @{\ \ }  c  @{\ \ }  c  @{\ \ }  c  @{\  \ } | c  @{\  \ } }
			\hline
			$ \nu $& $ \eta_h $&  \text{order}&  $ \bm{e_u} $& \text{order}& $ e_p $& \text{order}&  \text{eff}  \\ 
			\hline
			\multirow{5}{*}{$ 10^{-1} $}&  2.2724e-03& --&  2.2705e-03&  --&  4.8715e-05&  --&  0.9994 \\
			&  1.4571e-04&  2.06&   1.4570e-04&  2.06&  1.9981e-06&  2.39&   1.0000  \\
			&  9.1855e-06&  2.03&  9.1853e-06&  2.03&  8.4724e-08&  2.32&   1.0000 \\
			&  5.7611e-07&  2.02&  5.7611e-07&  2.02&  3.6568e-09&  2.29&  1.0000 \\
			&  3.6082e-08&  2.01&  3.6097e-08&  2.01&  1.5975e-10&  2.27&  1.0004 \\
			\hline
			\multirow{5}{*}{$ 10^{-3} $}&  2.2756e-02& --&  2.2737e-02&  --&  4.8543e-04&  --&  0.9994 \\
			&  1.4589e-03&  2.06&   1.4588e-03&  2.06&  1.9905e-05&  2.39&   1.0000  \\
			&  9.1963e-05&  2.03&  9.1961e-05&  2.03&  8.4371e-07&  2.32&   1.0000 \\
			&  5.7660e-06&  2.02&  5.7659e-06&  2.02&  3.6406e-08&  2.29&  1.0000 \\
			&  3.6086e-07&  2.01&  3.6086e-07&  2.01&  1.5886e-09&  2.27&  1.0000 \\
			\hline
			\multirow{5}{*}{$ 10^{-6} $}&  7.1962e-01& --&  7.1901e-01&  --&  1.5350e-02&  --&  0.9994 \\
			&  4.6135e-02&  2.06&   4.6132e-02&  2.06&  6.2944e-04&  2.39&   1.0000  \\
			&  2.9082e-03&  2.03&  2.9081e-03&  2.03&  2.6679e-05&  2.32&   1.0000 \\
			&  1.8234e-04&  2.02&  1.8234e-04&  2.02&  1.1512e-06&  2.29&  1.0000 \\
			&  1.1411e-05&  2.01&  1.1411e-05&  2.01&  5.0235e-08&  2.27&  1.0000 \\
			\hline
			\multirow{5}{*}{$ 10^{-10} $}&  7.1962e+01& --&  7.1901e+01&  --&  1.5350e+00&  --&  0.9994 \\
			&  4.6135e+00&  2.06&   4.6132e+00&  2.06&  6.2944e-02&  2.39&   1.0000  \\
			&  2.9082e-01&  2.03&  2.9081e-01&  2.03&  2.6679e-03&  2.32&   1.0000 \\
			&  1.8234e-02&  2.02&  1.8234e-02&  2.02&  1.1512e-04&  2.29&  1.0000 \\
			&  1.1411e-03&  2.01&  1.1411e-03&  2.01&  5.0235e-06&  2.27&  1.0000 \\
			\hline
		\end{tabular}
	\end{table}

	\subsection{Example 3}\label{num_example3}
	In this example, we consider the model problem \eqref{sec_model_eq1} with singular solution on an L-shape domain $ \Omega=\big((-1,1)\times(-1,1)\big)\setminus \big([0,1)\times(-1,0]\big) $ (see, \cite{1989_NM_posterioriStokes}). The domain and the corresponding initial mesh are displayed in Figure \ref{exam3_mesh1} and Figure \ref{exam3_mesh2}, respectively. Here, we select the $ k=1, 2, 3, 4 $, $ \nu=1 $ and $ \bm{f}=\bm{0} $. Let $ (r, \theta) $ denote the system of polar coordinates, then, the velocity and pressure are set to be 
	\begin{align*}
		\bm{u} &=  r^\lambda \begin{pmatrix}
			(1+\lambda)\sin(\theta)\Psi(\theta) + \cos(\theta)\Psi'(\theta) \\
			\sin(\theta)\Psi'(\theta) - (1+\lambda)\cos(\theta)\Psi(\theta)
		\end{pmatrix}, \\
		p &= \frac{-r^{\lambda-1}\Big( (1+\lambda)^2\Psi'(\theta) + \Psi'''(\theta) \Big)}{ (1-\lambda)} ,
	\end{align*}
	where
	\begin{align*}
		&\Psi(\theta) = \frac{\sin((1+\lambda)\theta)\cos(\lambda \omega)}{1+\lambda} - \cos((1+\lambda)\theta) \\
		&\qquad\quad -\frac{\sin((1-\lambda)\theta)\cos(\lambda \omega)}{1-\lambda} + \cos((1-\lambda)\theta), \\
		& \omega = \frac{3\pi}{2}, \quad \lambda = \frac{856399}{1572564}.
	\end{align*}
	
	We know that $ (\bm{u}, p) $ is analytic in $ \overline{\Omega}\setminus \{(0,0)\} $ and, $ (\nabla\bm{u}, p) $ are singular at the origin. This example reflects the typical singular behavior of the solution of two-dimensional Stokes problem near the reentrant corners of the computational domain. We use a standard adaptive algorithm to resolve the singularity:
	\begin{center}
		Solve $ \rightarrow $ Estimate $ \rightarrow $ Mark $ \rightarrow $ Refine.
	\end{center}
	Given the initial mesh, we execute the above procedure to obtain a new mesh, which is called one iteration. This process is stopped until the estimator $ \eta_h $ is less than the stopping criterion $ tol $. 
	
	In this example, we take $ tol $ = 0.01. For the different polynomial orders, the $ \eta_h $, the number of iterations and degrees of freedom required to reach the tolerance are respectively displayed in Table \ref{numExam3_tb_1}, the refined meshes are plotted in Figure \ref{exam3_refine}. We see that the refinements are focused on the origin, i.e., the singularity at the origin is successfully captured by the refinement. The convergence results are reported in Figure \ref{exam3_rate}, where we observe that the error and estimator achieve the $ (k+1)/2 $ order convergence with respect to the degrees of freedom.
	
	\begin{figure}[!ht]
		\begin{center}
			\subfigure[L-shaped domain $ \Omega $ ]{
				\label{exam3_mesh1}
				\centering
				\includegraphics[width=2in]{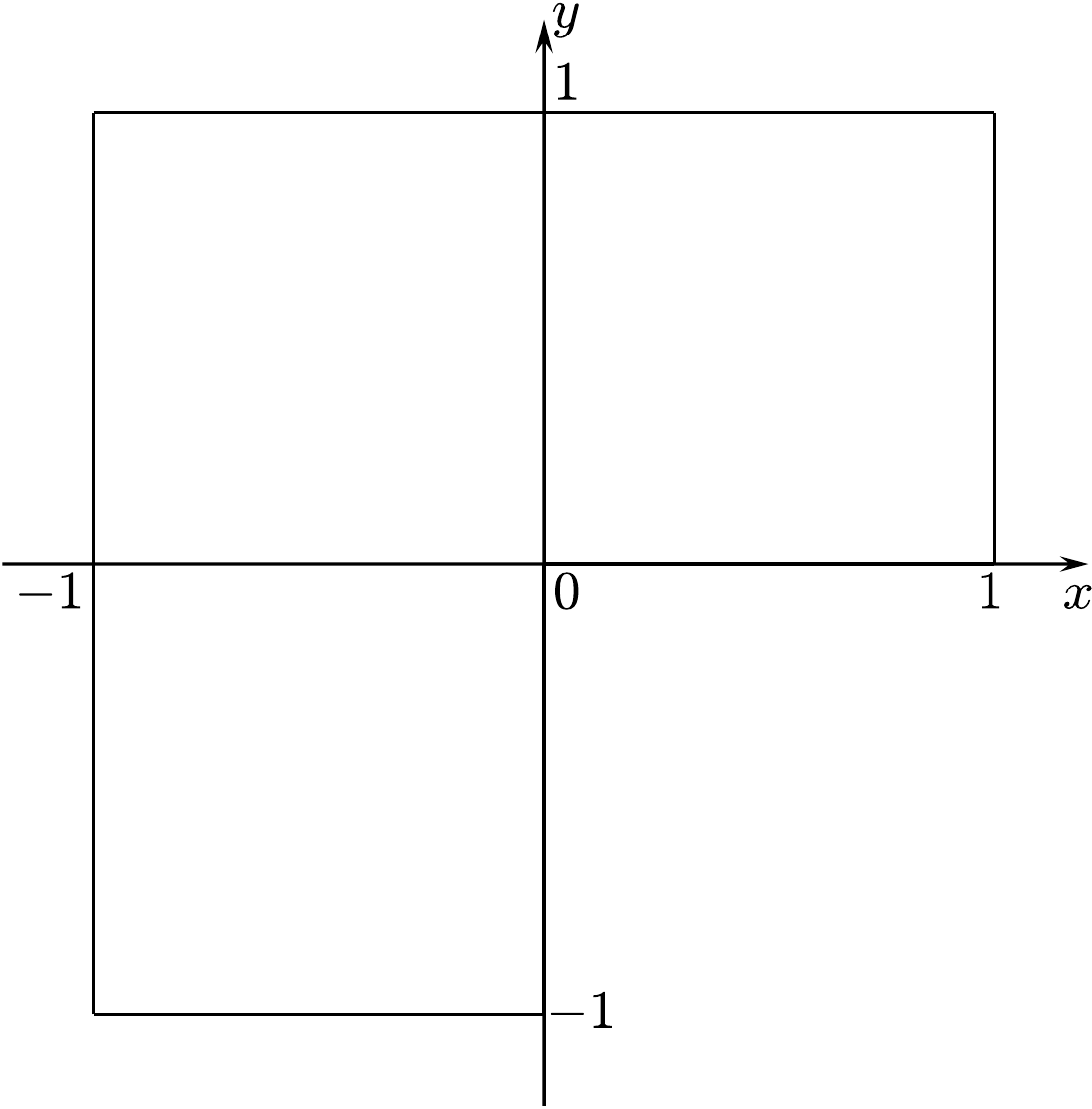}      
			}
			\qquad\qquad
			\subfigure[The $ \mathcal{T}_{h}^{4} $-like mesh]{
				\label{exam3_mesh2}
				\centering
				\includegraphics[scale=0.4,trim=120 43 110 45,clip]{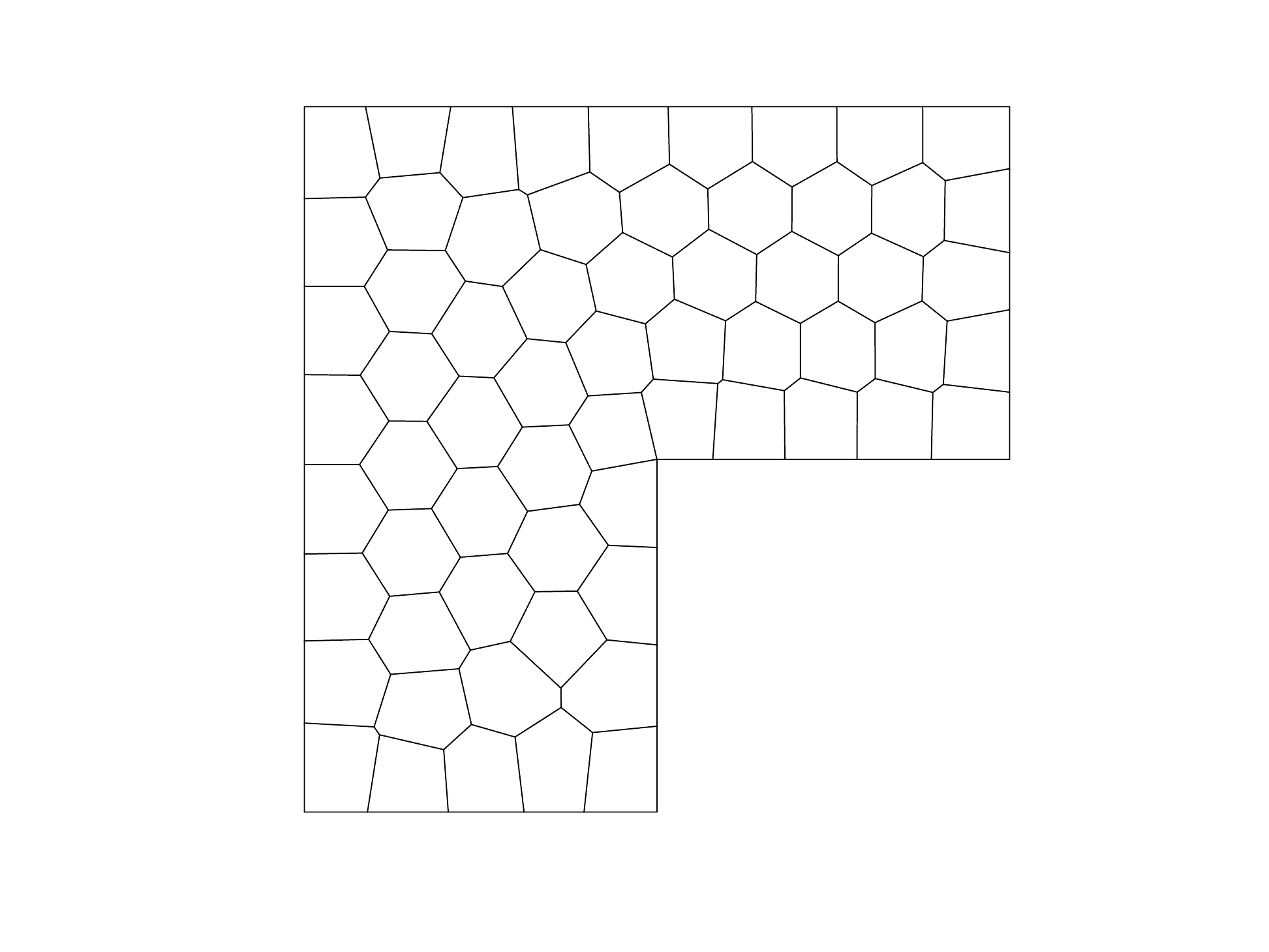}
			}
		\end{center}
		\vspace{-1.5em}
		\caption{Example 3: Illustrations of domain and mesh.} \label{exam3_domainmesh}
	\end{figure}

	\begin{table}[h!]
		\caption{Example 3: The required number of iterations to reach the tolerance $ tol=0.01 $.}
		\label{numExam3_tb_1}
		\medskip
		\centering
		\begin{tabular}[b]{ @{\  \ }c@{\ \ }| @{\ \ }  c  @{\ \ }  c  @{\ \ }  c  @{\ \ }  c  @{\ \ }  c  @{\ \ }  c  @{\  \ } }
			\hline
			$ k $& $ \eta_h $&  $ \text{Iterations} $&  $ \sharp\text{Dof} $ \\ 
			\hline
			1&  9.0836e-03&  25&   97126  \\
			\hline
			2&  7.9897e-03&  26&   19032  \\
			\hline
			3&  8.9269e-03&  28&   11108  \\
			\hline
			4&  8.4906e-03&  31&   10370  \\
			\hline
		\end{tabular}
	\end{table}

	\begin{figure}[!ht]
		\begin{center}
			\subfigure[$ k=1 $]{
				\label{exam3_refine_1}
				\centering
				\includegraphics[scale=0.5,trim=120 60 110 40,clip]{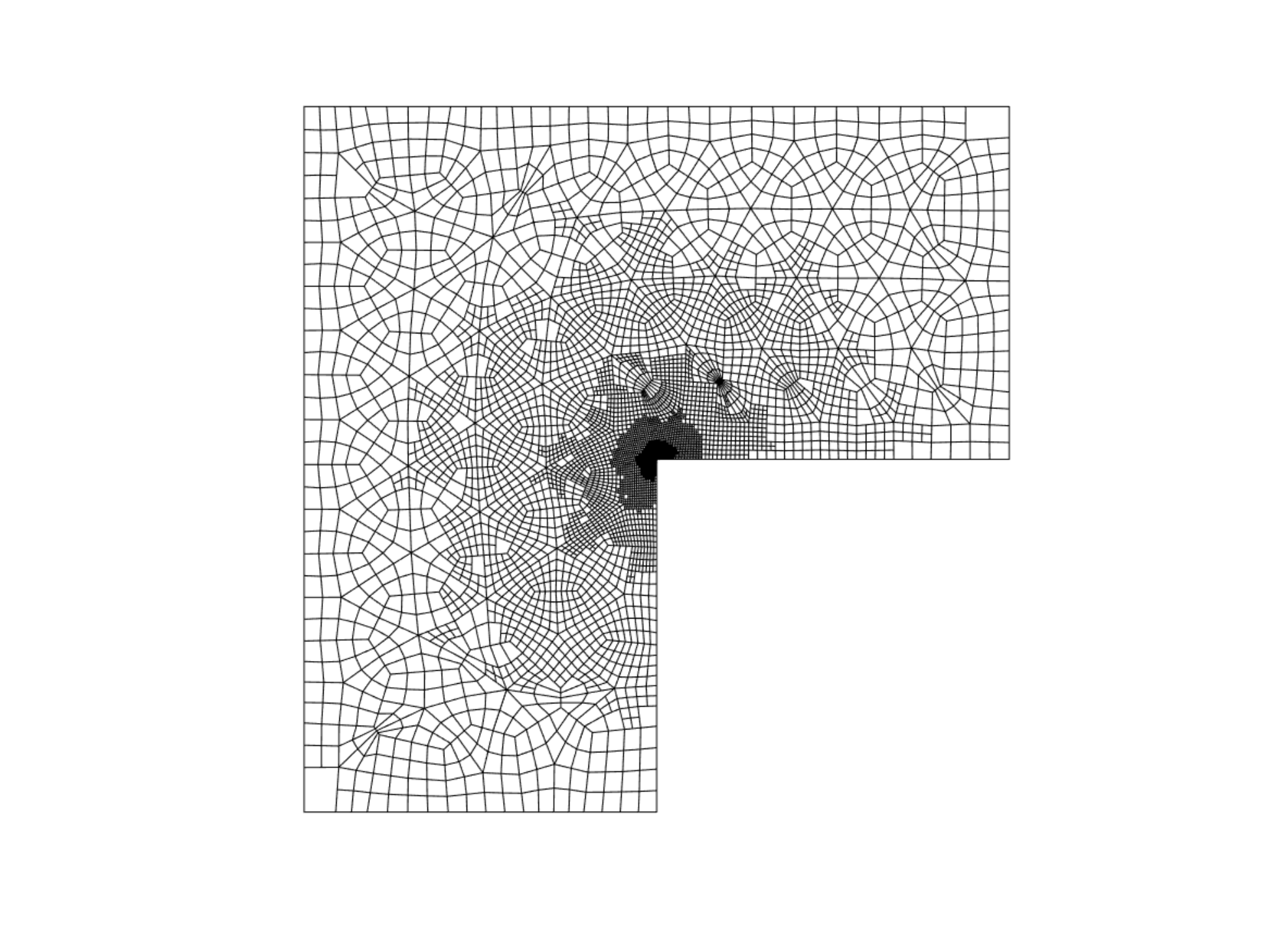}      
			}
			\subfigure[$ k=2 $ ]{
				\label{exam3_refine_2}
				\centering
				\includegraphics[scale=0.5,trim=120 60 110 40,clip]{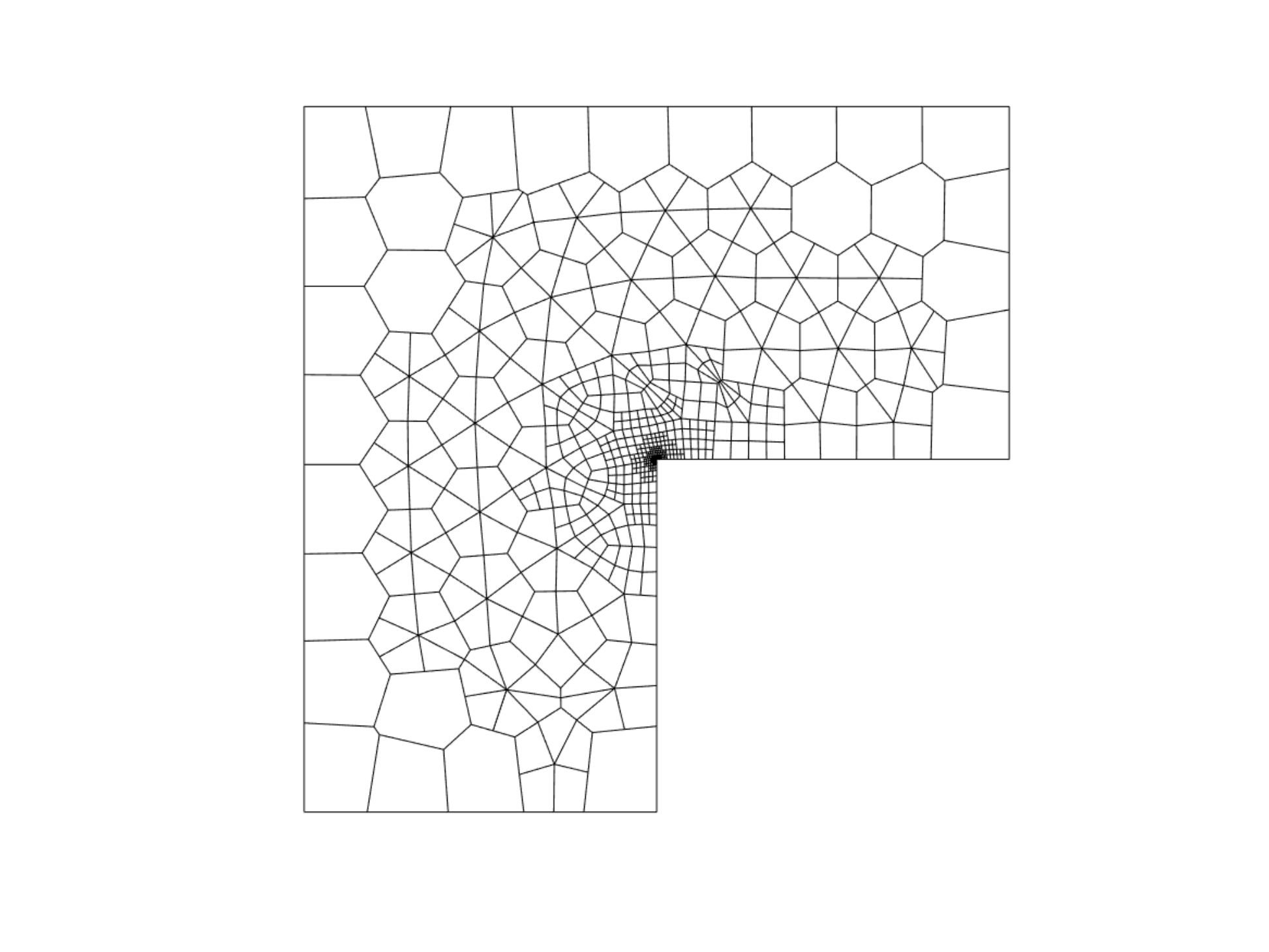}     
			}
			\subfigure[$ k=3 $]{
				\label{exam3_refine_3}
				\centering
				\includegraphics[scale=0.5,trim=120 60 110 40,clip]{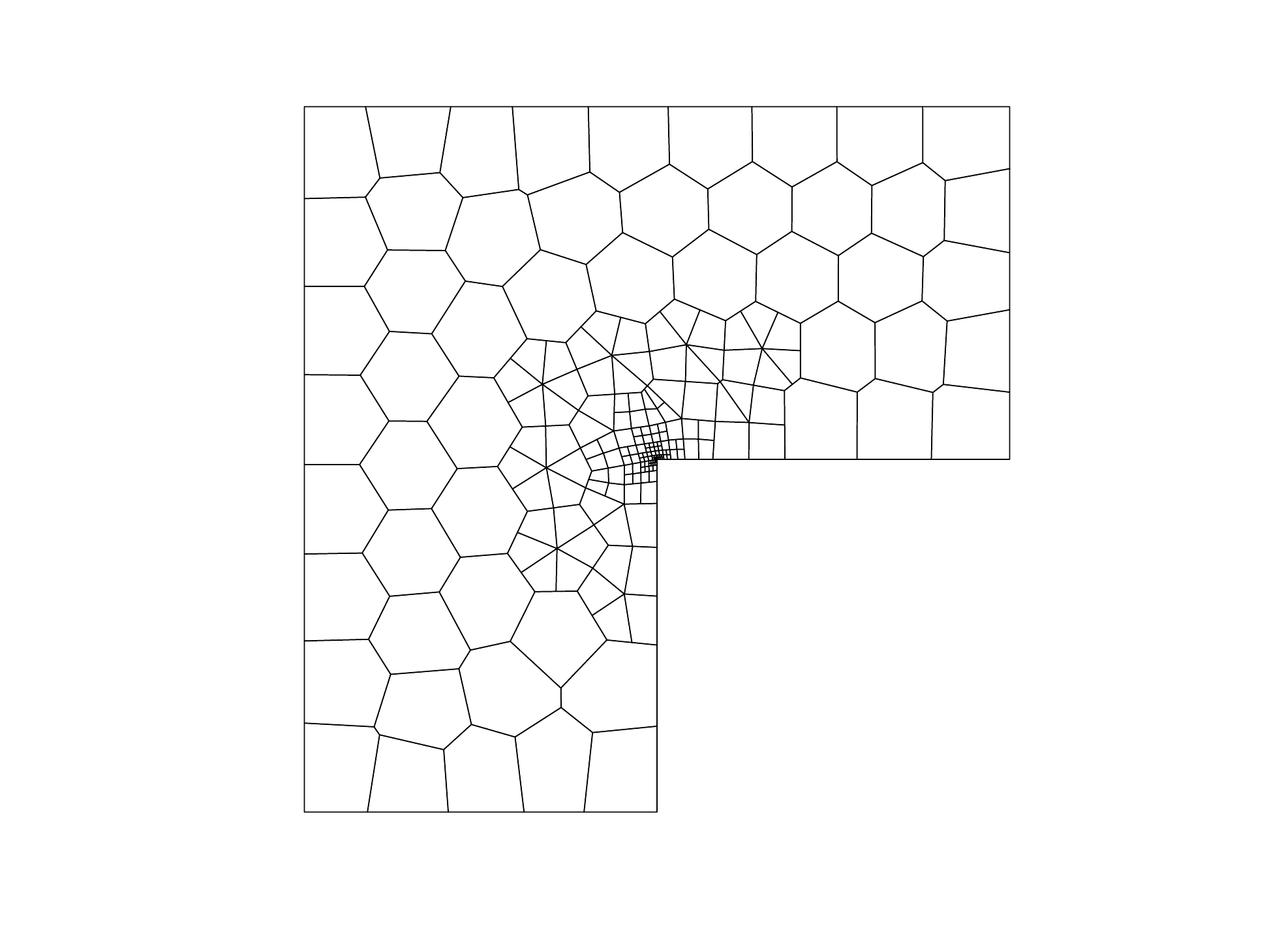}     
			}
			\subfigure[$ k=4 $]{
				\label{exam3_refine_4}
				\centering
				\includegraphics[scale=0.5,trim=120 60 110 40,clip]{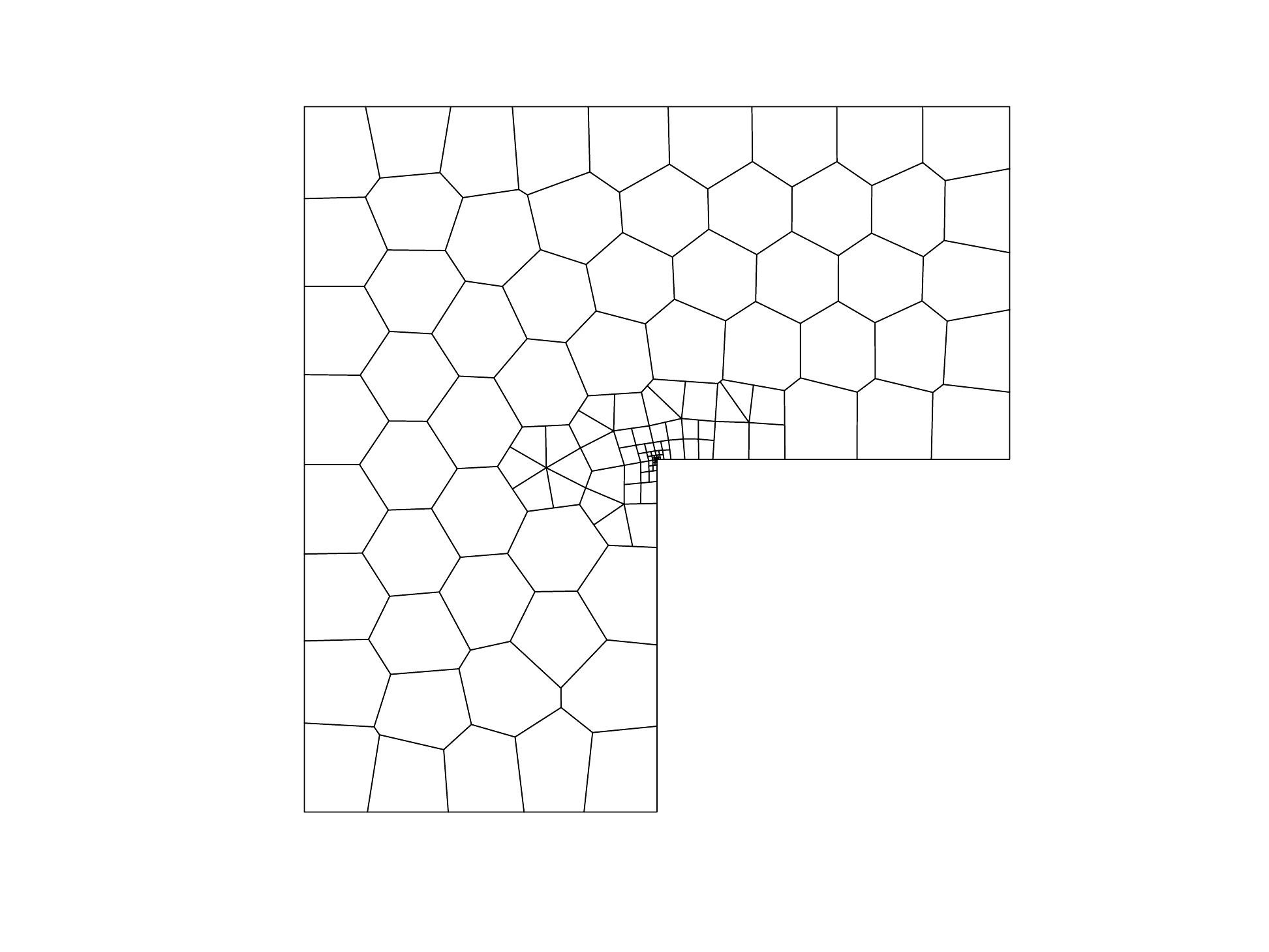}      
			}
		\end{center}
		\vspace{-1.5em}
		\caption{Example 3: The finally refined meshes for different polynomial orders.} \label{exam3_refine}
	\end{figure}

	\begin{figure}[!ht]
		\begin{center}
			\subfigure[$ k=1 $]{
				\label{exam3_rate_1}
				\centering
				\includegraphics[width=3in]{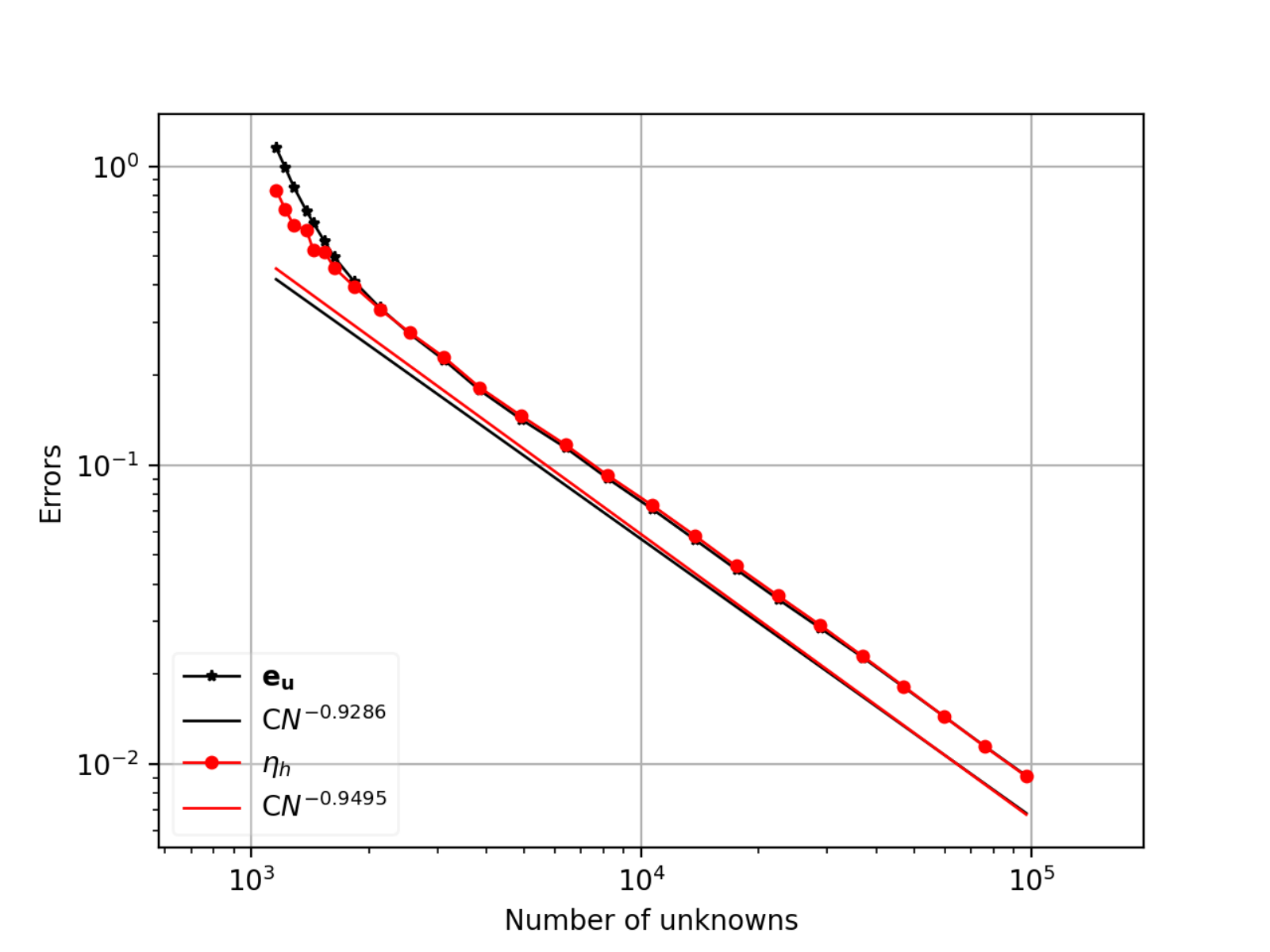}      
			}
			\subfigure[$ k=2 $ ]{
				\label{exam3_rate_2}
				\centering
				\includegraphics[width=3in]{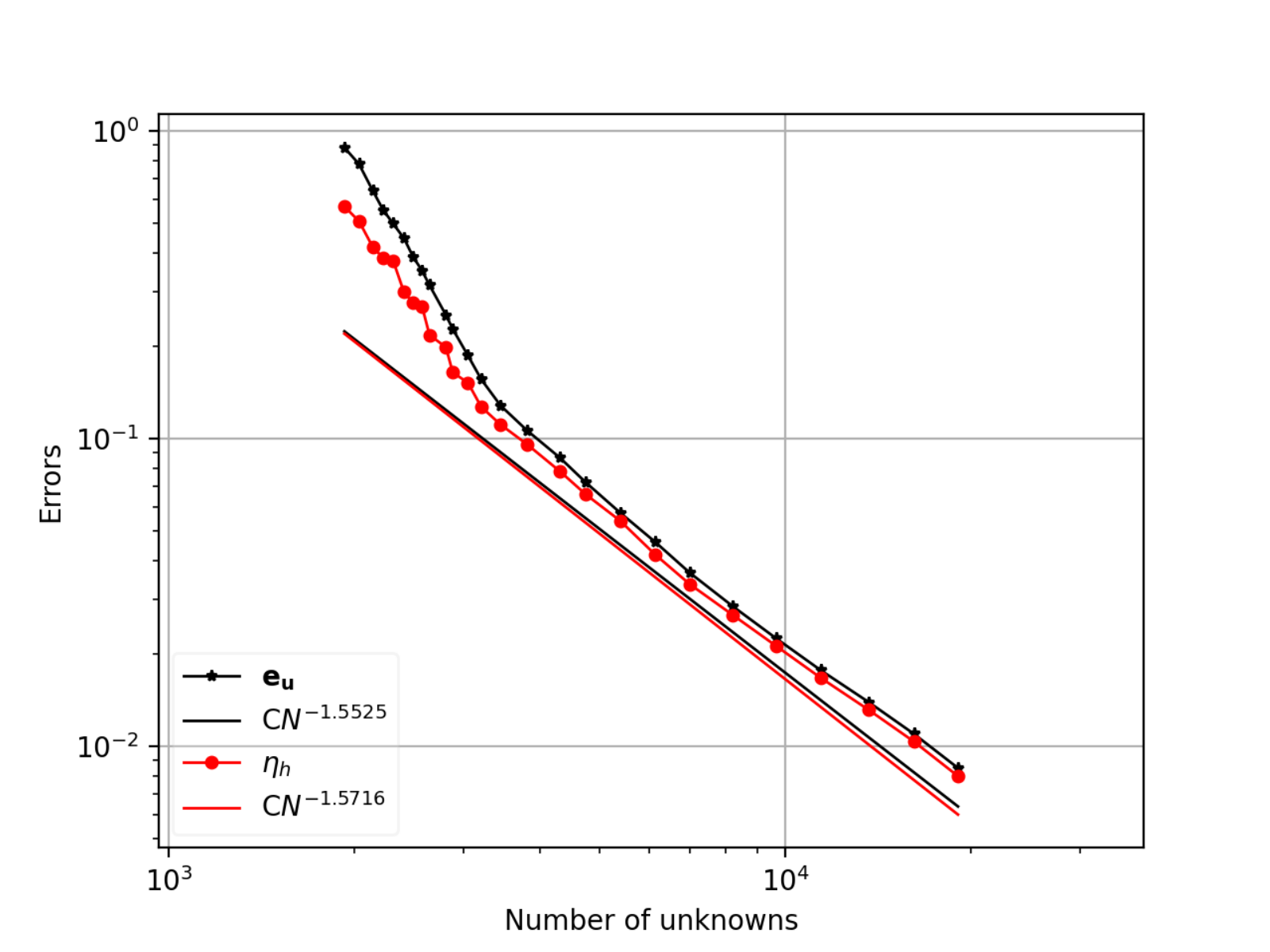}     
			}
			\subfigure[$ k=3 $]{
				\label{exam3_rate_3}
				\centering
				\includegraphics[width=3in]{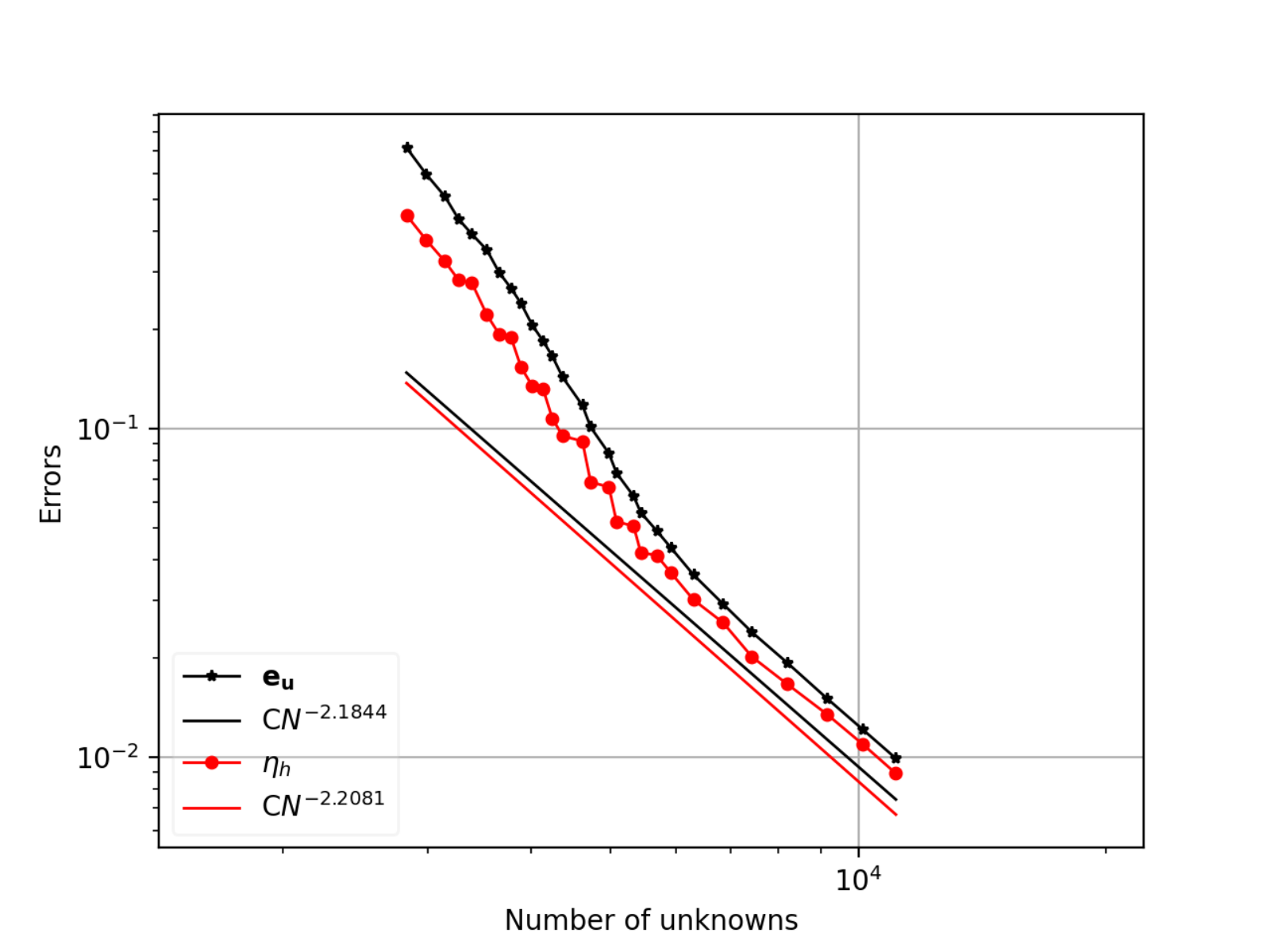}     
			}
			\subfigure[$ k=4 $]{
				\label{exam3_rate_4}
				\centering
				\includegraphics[width=3in]{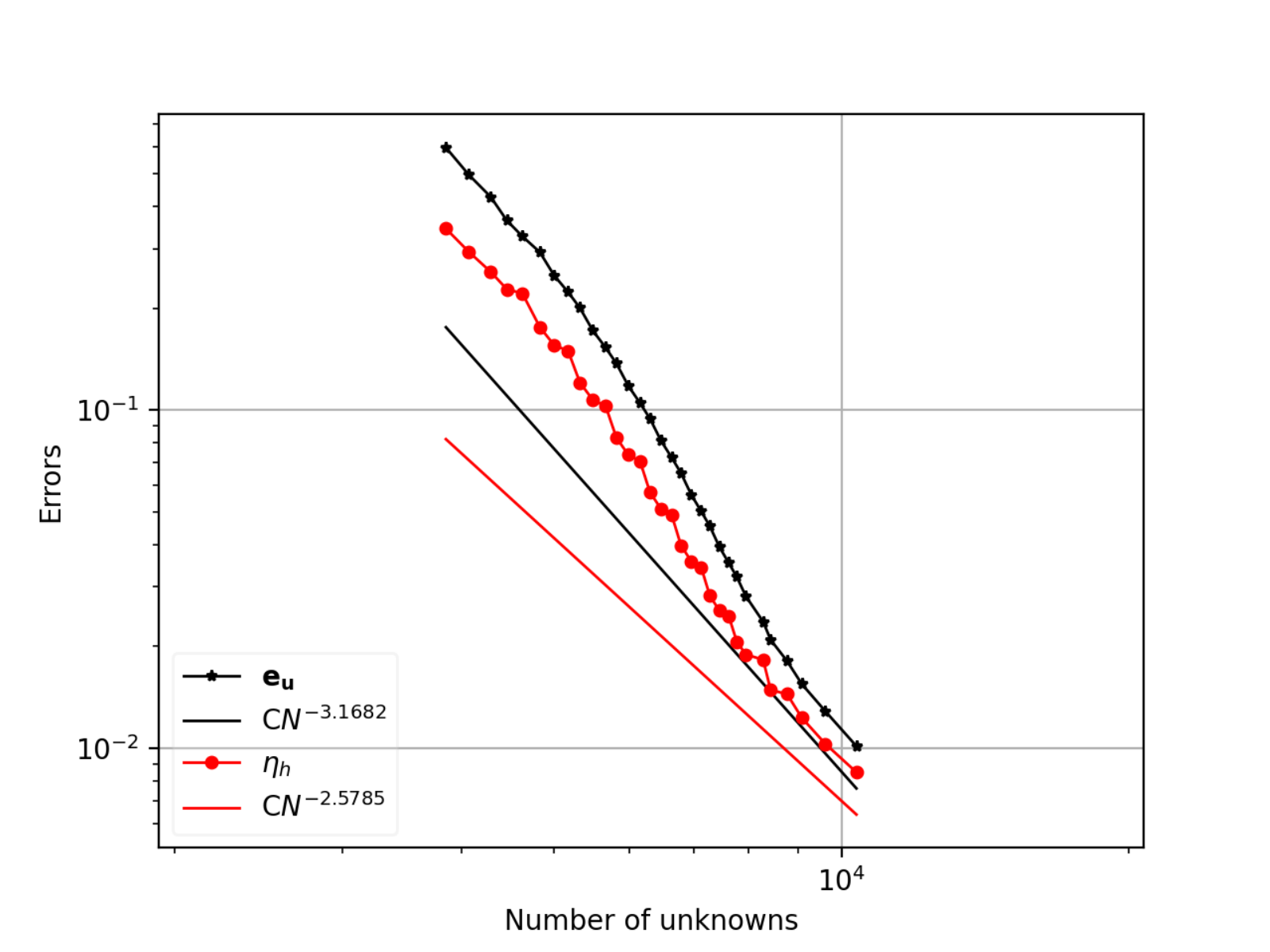}      
			}
		\end{center}
		\vspace{-1.5em}
		\caption{Example 3: The convergence results for different polynomial orders.} \label{exam3_rate}
	\end{figure}

	\subsection{Example 4}\label{num_example4}
	In this example, we consider the problem of flow around a circular cylinder offset slightly in a channel, the geometry and the $ \mathcal{T}_{h}^{4} $-like initial mesh are displayed in Figure \ref{exam4_domainmesh}. The left vertical side and the right vertical side of the whole domain are the inﬂow boundary and the outﬂow boundary, respectively. Both the inflow and outflow profile are set as
	\begin{align*}
		\bm{u} = \begin{pmatrix}
			\frac{6}{0.41^2}\sin(\frac{\pi}{8})y(0.41-y) \\
			0
		\end{pmatrix}.
	\end{align*}
	No-slip boundary conditions are prescribed along the top and bottom boundaries. In this fluid flow, where the cylinder is an obstacle, the velocity of the fluid around the cylinder varies considerably, so we expect our estimator can provide a guidance for mesh refining.
	
	Here, we take the polynomial order $ k=1 $ and the viscosity $ \nu=1 $. We start from the initial mesh, see Figure \ref{exam4_mesh}, and the required 8 iterations to reach the stopping criterion $ tol=0.15 $. The successive refinements are plotted in Figure \ref{exam4_refine_mesh}. It can be seen from these refined meshes that the local mesh refinement is carried out near the cylinder, which means that our estimator can be effective in capturing the region where the velocity changes drastically. Finally, we show the streamline of the velocity in Figure \ref{exam4_meshvelocity}.
	
	\begin{figure}[!ht]
		\begin{center}
			\subfigure[The domain $ \Omega $ ]{
				\label{exam4_domain}
				\centering
				\includegraphics[scale=0.7]{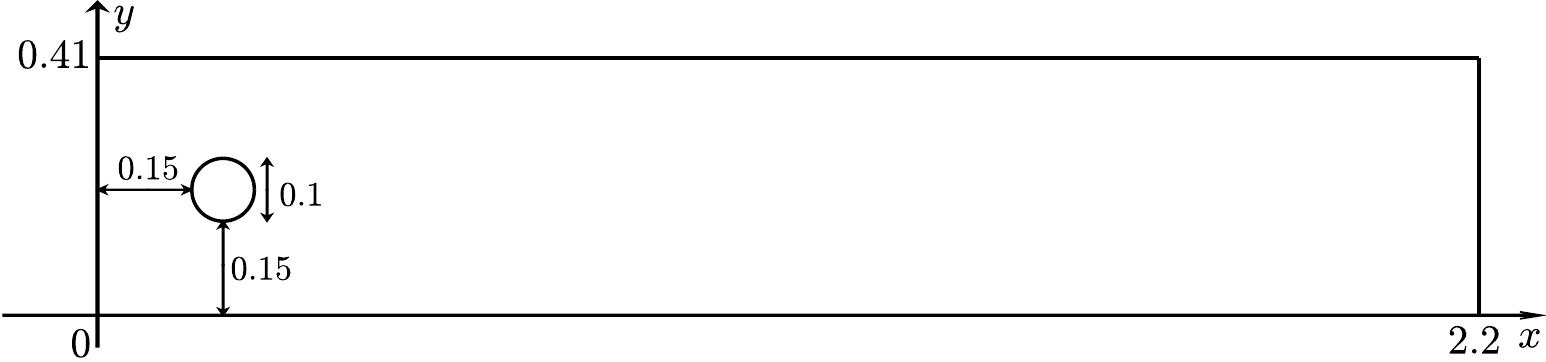}      
			}
			\qquad\qquad
			\subfigure[The $ \mathcal{T}_{h}^{4} $-like mesh]{
				\label{exam4_mesh}
				\centering
				\includegraphics[scale=0.7,trim=70 160 60 160,clip]{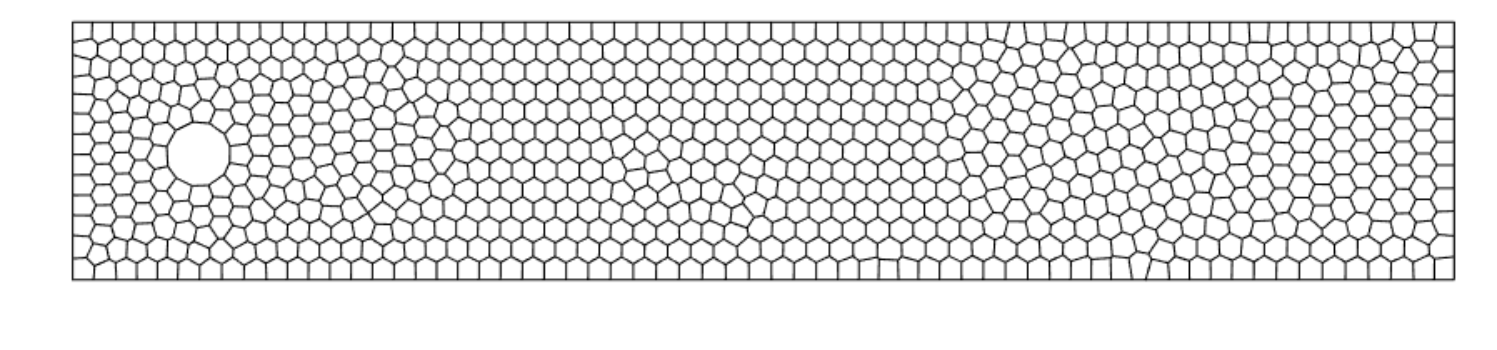}
			}
		\end{center}
		\vspace{-1.5em}
		\caption{Example 4: Illustrations of domain and mesh.} \label{exam4_domainmesh}
	\end{figure}

	\begin{figure}[!ht]
		\begin{center}
			\subfigure[]{\label{exam4_refine_mesh_1}\includegraphics[scale=0.55,trim=70 165 70 165,clip]{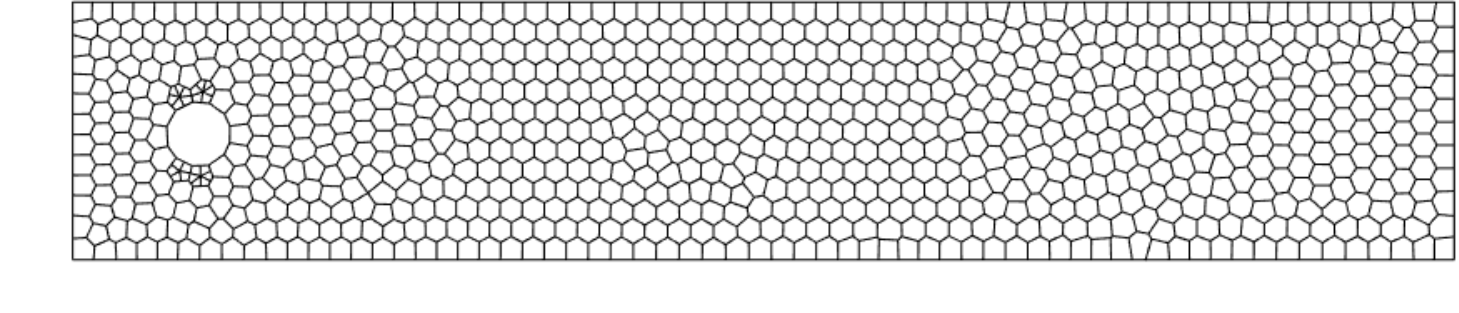} }
			\subfigure[]{ \label{exam4_refine_mesh_2} \includegraphics[scale=0.55,trim=70 165 70 165,clip]{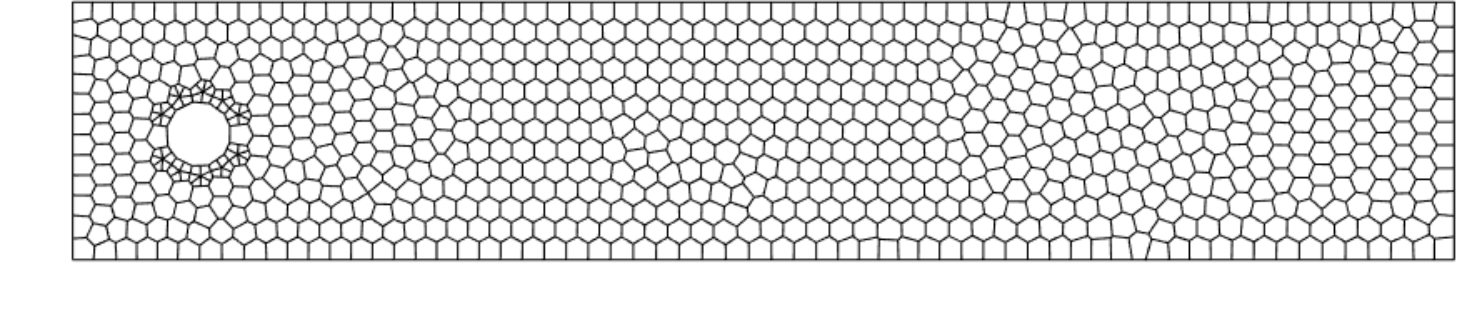} }
			\subfigure[]{\label{exam4_refine_mesh_3} \includegraphics[scale=0.55,trim=70 165 70 165,clip]{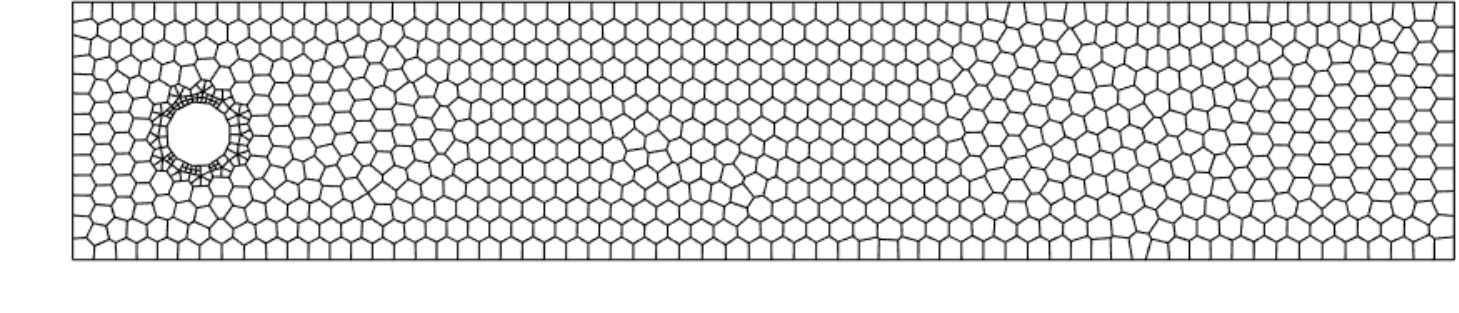} }
			\subfigure[]{ \label{exam4_refine_mesh_4} \includegraphics[scale=0.55,trim=70 165 70 165,clip]{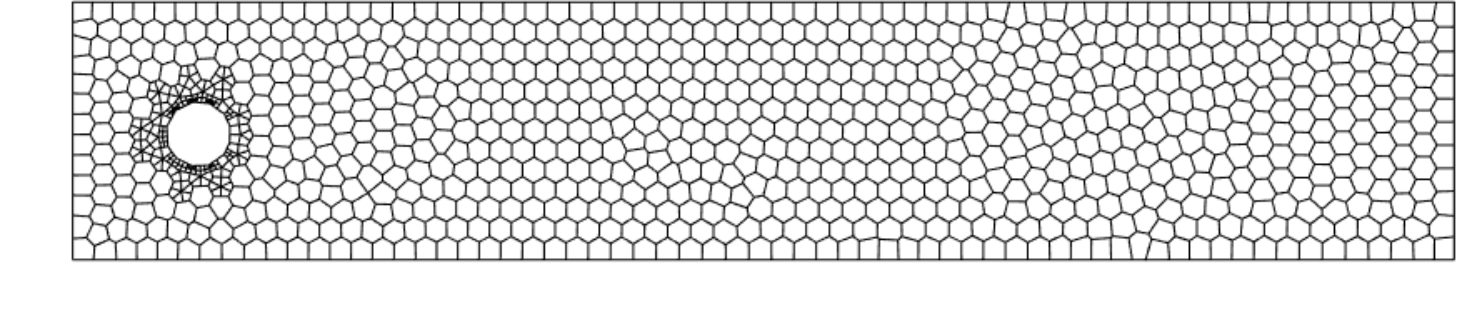} }
		\end{center}
		\vspace{-1.5em}
		\caption{Example 4: Meshes after 1, 3, 5 and 8 iterations, respectively.} \label{exam4_refine_mesh}
	\end{figure}

	\begin{figure}[!ht]
		\begin{center}
			\includegraphics[scale=0.6]{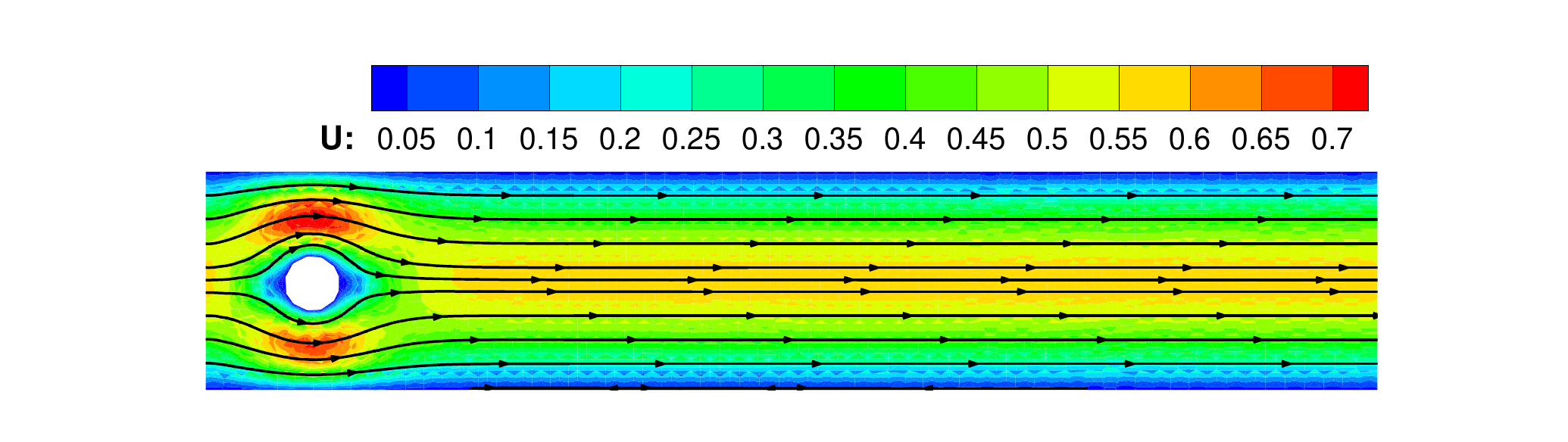}
		\end{center}
		\vspace{-2em}
		\caption{Example 4: The velocity and the streamline.} \label{exam4_meshvelocity}
	\end{figure}

	\subsection{Example 5}\label{num_example5}
	In this test, we take the polynomial order $ k=2 $, the viscosity $ \nu=1 $ and, let the domain be $ \Omega = (-2, 2)\times (-1,0) $, where the geometry and the $ \mathcal{T}_h^2 $-like mesh are shown in Figure \ref{exam5_domainmesh}. We consider the Stokes problem from \citep[Step 22]{dealII92}, which relates to a problem in geophysics that we want to compute the flow field of magma in the earth's interior under a mid-ocean rift. Rifts are places where two continental plates are very slowly drifting apart (a few centimeters per year at most), leaving a crack in the earth crust that is filled with magma from below. Without trying to be entirely realistic, we model this situation by taking the source term $ \bm{f}=\bm{0} $ and setting the following boundary conditions:
	\begin{align*}
		\bm{u} = \begin{pmatrix}
			-1 \\ 
			0
		\end{pmatrix} \text{ at } y =0,\ x < 0; \quad
		\bm{u} = \begin{pmatrix}
			1 \\ 
			0
		\end{pmatrix} \text{ at } y =0,\ x > 0; \quad
		\bm{u} = \begin{pmatrix}
			0 \\ 
			0
		\end{pmatrix} \text{ at } y =0,\ x = 0,
	\end{align*}
	and using natural boundary conditions $ [p\mathbf{I} - \nabla \bm{u}]\bm{n} = \bm{0} $ everywhere else. By the settings of boundary conditons, we expect that the flow field will pull material from below and move it to the left and right ends of the domain. The discontinuity of velocity boundary conditions will produce a singularity in the pressure at the center of the upper boundary that sucks material all the way to the upper boundary to fill the gap left by the outward motion of material at this location.
	
	After 15 iterations, the refined meshes are shown in Figure \ref{exam5_refine_mesh}. What's more, Figure \ref{exam5_meshvelocity} gives the velocity and the corresponding streamline, which shows that the fluid transported along with the moving upper boundary and being replaced by material coming from below. Observe how the grid is refined in regions where the solution rapidly changes: On the upper boundary, we have Dirichlet boundary conditions that are $ -1 $ in the left half of the line and $ 1 $ in the right one, so there is an abrupt change at $ x=0 $. Likewise, there are changes from Dirichlet to Neumann data in the two upper corners, so there is need for refinement there as well, but here the change in velocity is not as dramatic as at $ x=0 $, as can also be seen in Figure \ref{exam5_meshvelocity}.

	\begin{figure}[!ht]
		\begin{center}
			\subfigure[The domain $ \Omega $ ]{
				\label{exam5_domain}
				\centering
				\includegraphics[scale=0.6,trim=70 140 50 130,clip]{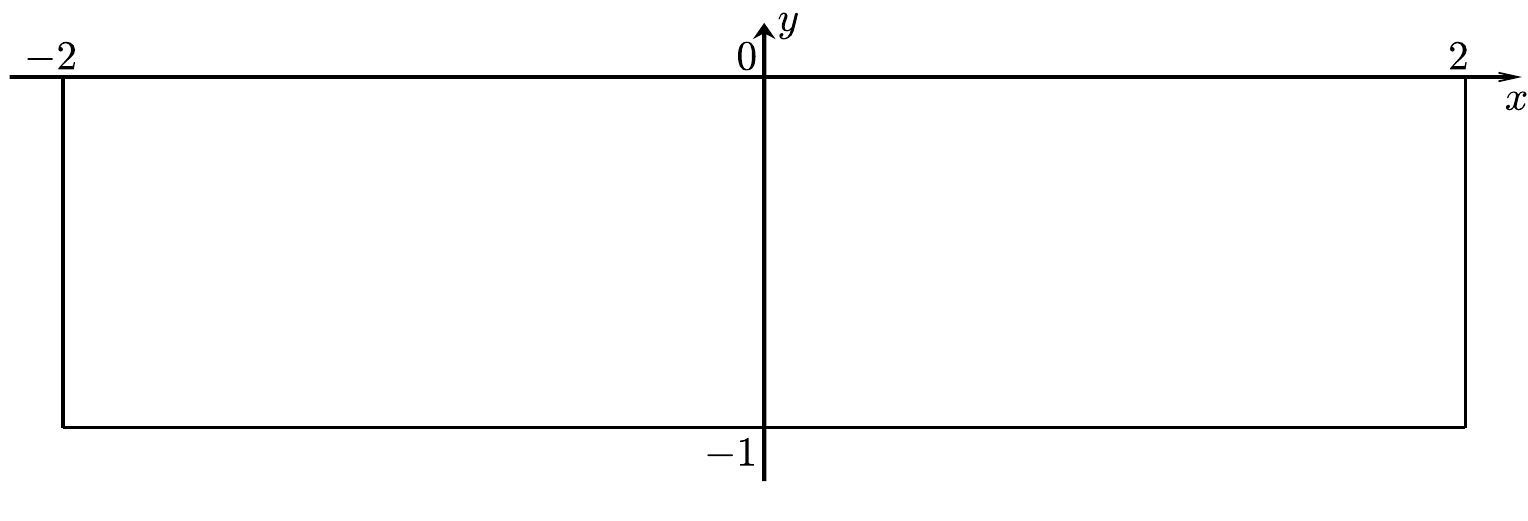}      
			}
			\qquad\qquad
			\subfigure[The $ \mathcal{T}_{h}^{2} $-like mesh]{
				\label{exam5_mesh}
				\centering
				\includegraphics[scale=0.59,trim=78 160 60 140,clip]{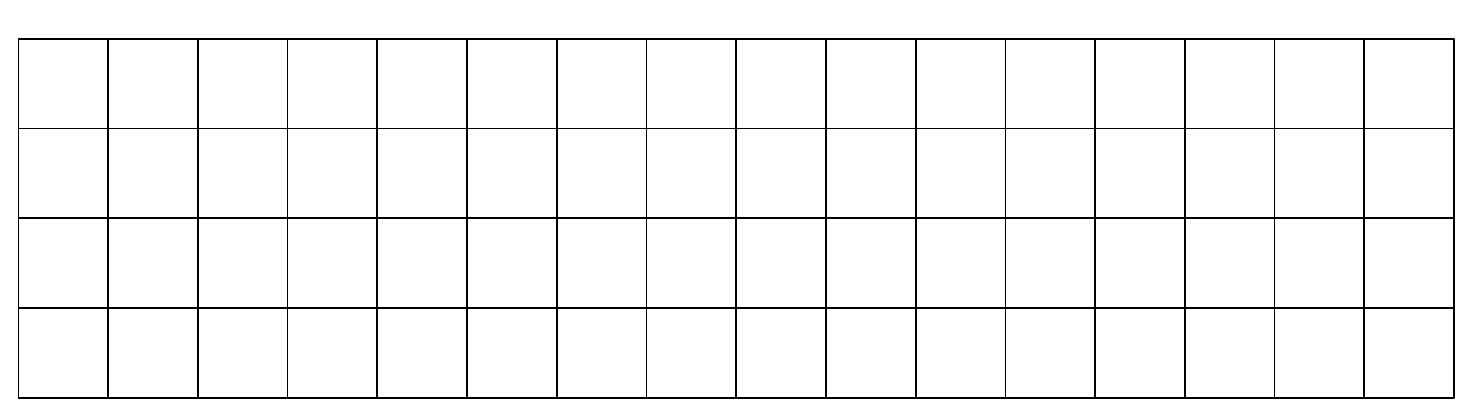}
			}
		\end{center}
		\vspace{-1.5em}
		\caption{Example 5: Illustrations of domain and mesh.} \label{exam5_domainmesh}
	\end{figure}

	\begin{figure}[!ht]
		\begin{center}
			\subfigure[]{ \label{exam5_refine_mesh_1}\includegraphics[scale=0.53,trim=65 145 60 150,clip]{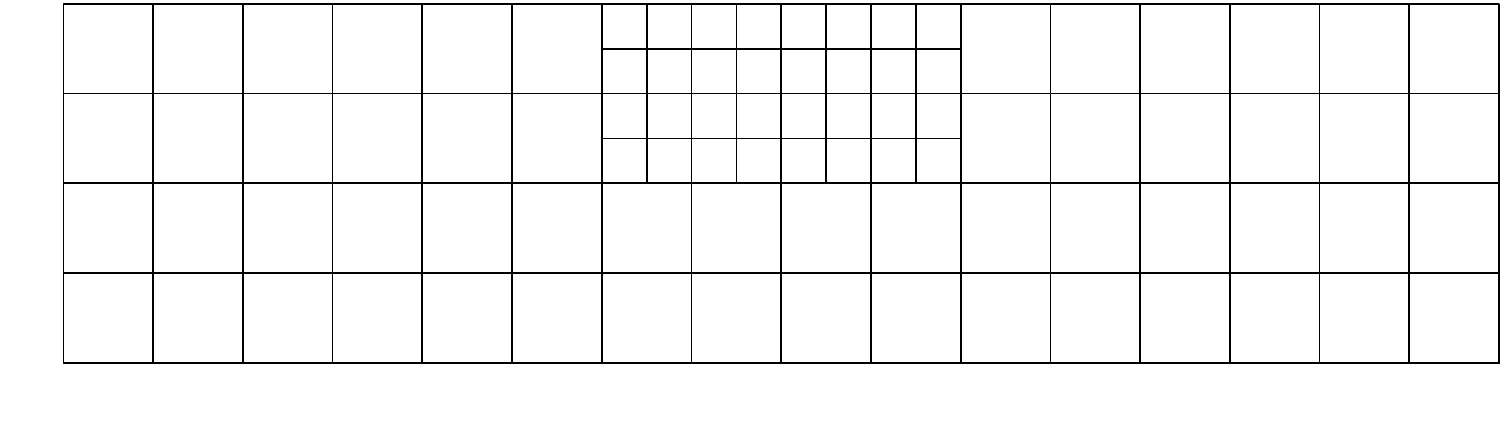}}
			\subfigure[]{ \label{exam5_refine_mesh_2}\includegraphics[scale=0.53,trim=65 145 60 150,clip]{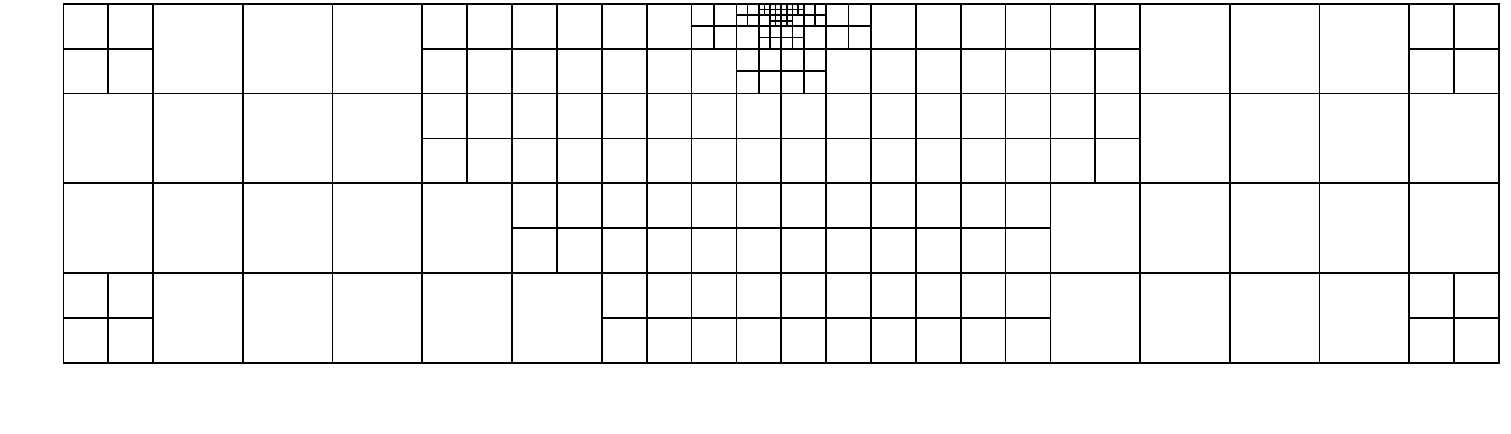}}
			\subfigure[]{ \label{exam5_refine_mesh_3}\includegraphics[scale=0.53,trim=65 145 60 150,clip]{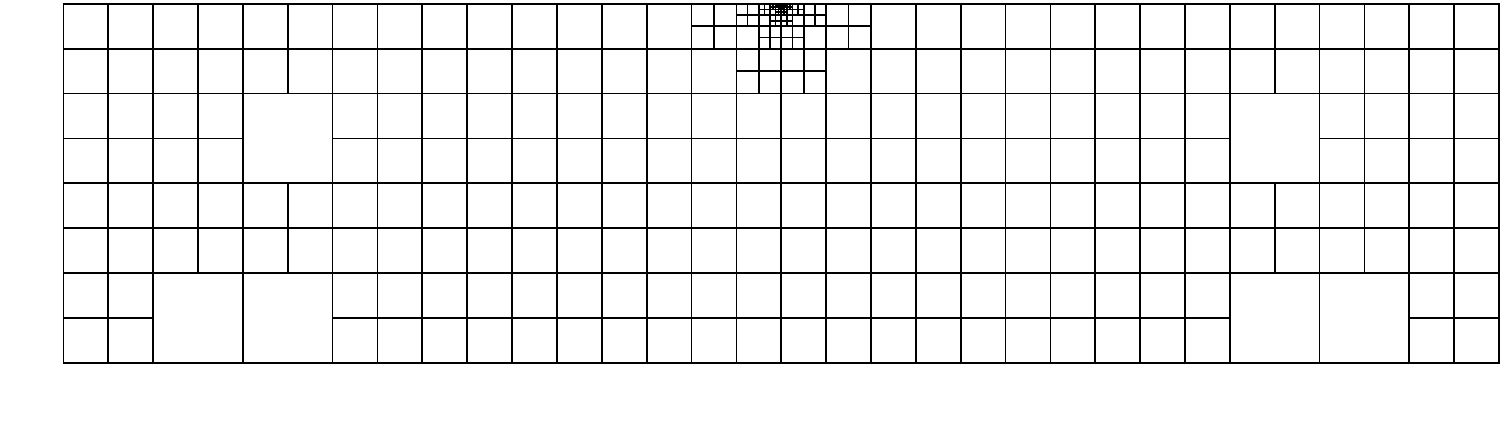}}
			\subfigure[]{ \label{exam5_refine_mesh_4}\includegraphics[scale=0.53,trim=65 145 60 150,clip]{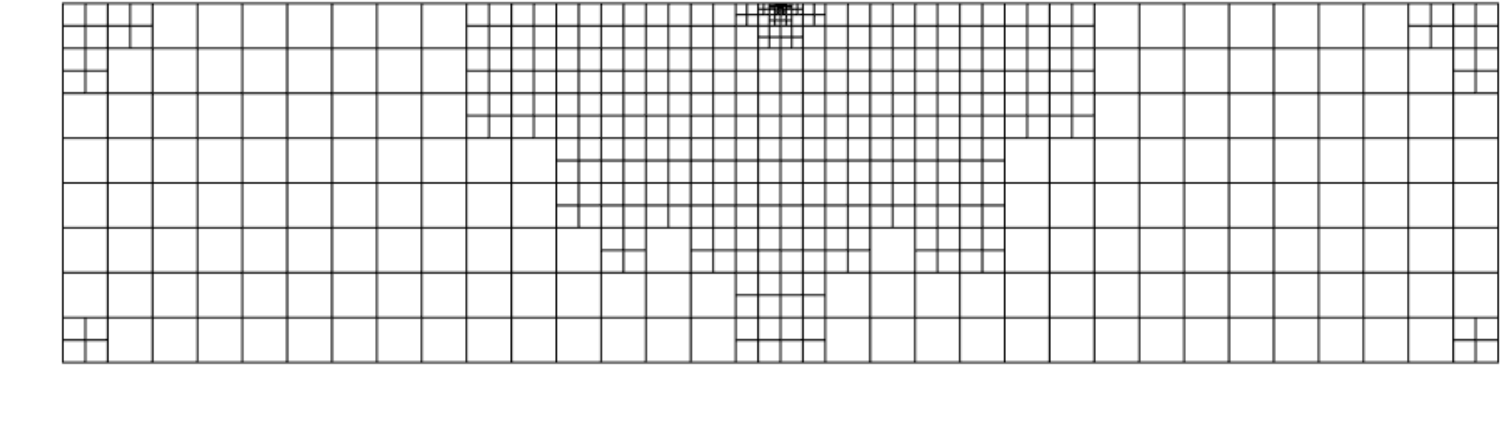}}
			\subfigure[]{ \label{exam5_refine_mesh_5}\includegraphics[scale=0.53,trim=65 145 60 150,clip]{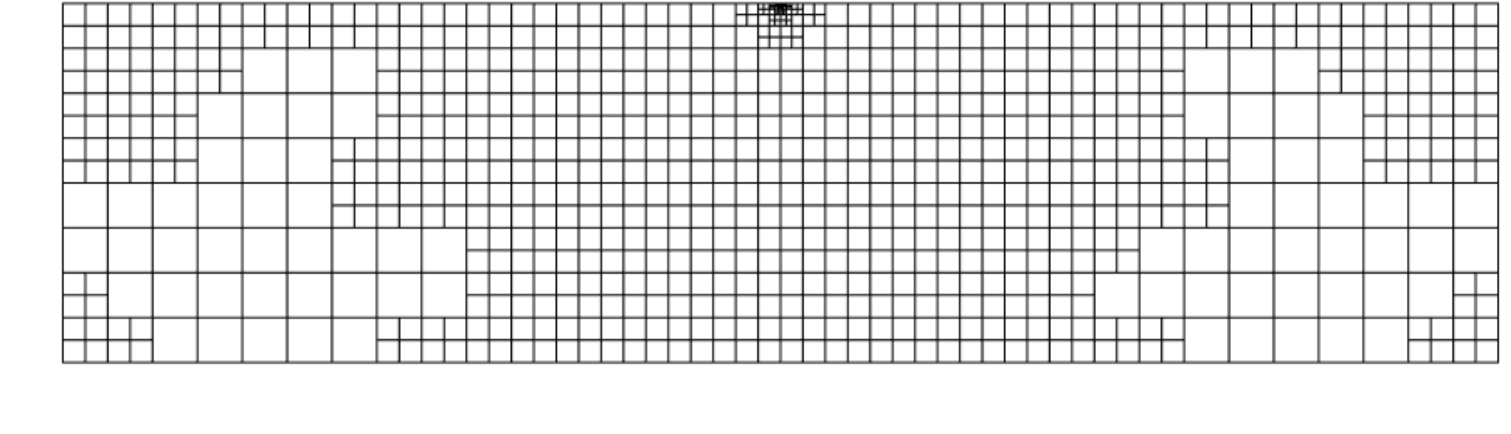}}
			\subfigure[]{ \label{exam5_refine_mesh_6}\includegraphics[scale=0.53,trim=65 145 60 150,clip]{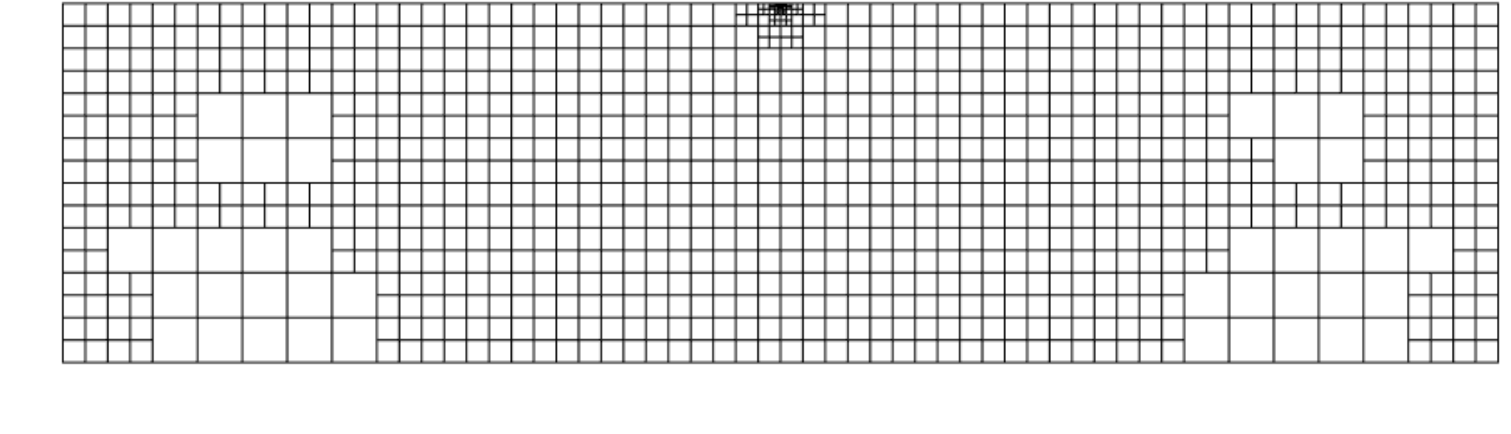}}
		\end{center}
		\vspace{-1.5em}
		\caption{Example 5: Meshes after 1, 4, 7, 13, and 15 iterations, respectively.} \label{exam5_refine_mesh}
	\end{figure}
	
	\begin{figure}[!ht]
		\begin{center}
			\includegraphics[scale=0.7]{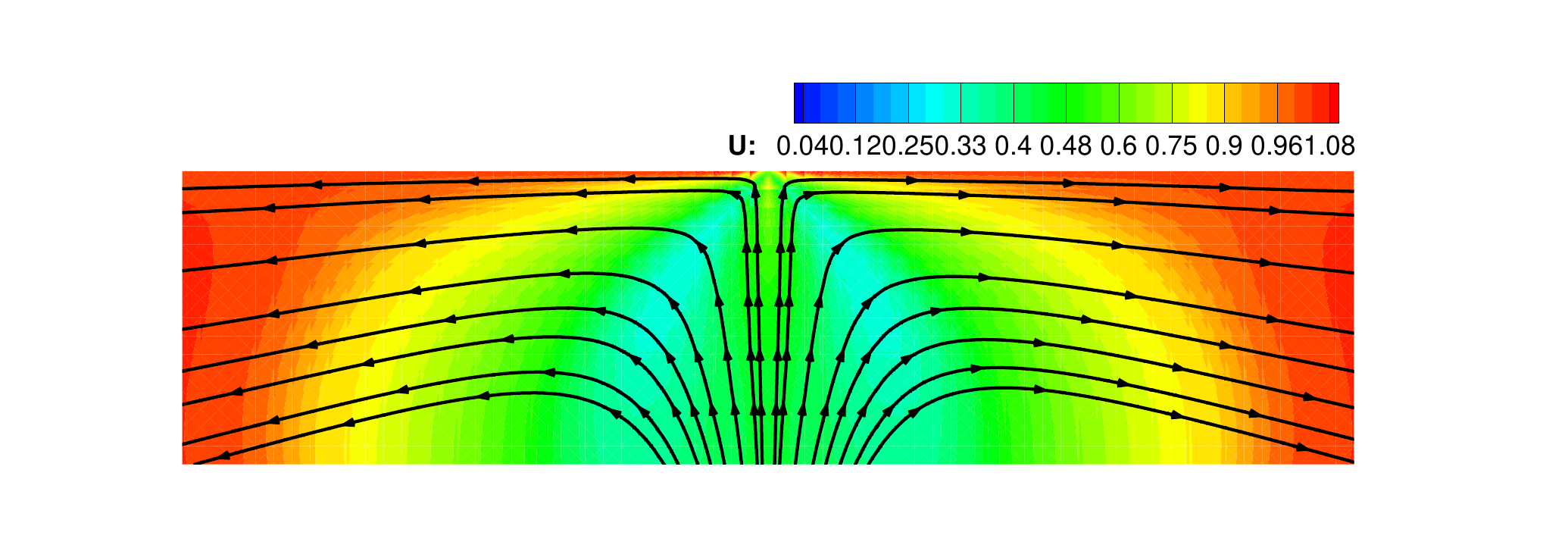}
		\end{center}
		\vspace{-2.5em}
		\caption{Example 5: The velocity and the streamline.} \label{exam5_meshvelocity}
	\end{figure}
	
	\section{Conclusion}\label{Conclusion}
	In this paper, we have presented a residual type a posteriori error estimator for the hybrid high-order (HHO) method for the Stokes problem. It is proved that the proposed estimator has the upper bound and lower bound, and this leads to the final adaptive algorithm of HHO method for the Stokes problem. The HHO method and the estimator allow the use of general meshes and support arbitrary approximation orders, which simplifies the procedure of adaptive mesh refinement and makes it easy to obtain high order computational accuracy. Some numerical examples are reported to illustrate the good performance of our estimator in the adaptive algorithm.

	\section*{Acknowledgements}
	The authors should thank Huayi Wei from Xiangtan University, China, for the valuable discussions of the codes in FEALPy.

	\setlength{\bibsep}{0.5ex}  
	\bibliography{yc_ref}
	
\end{document}